\documentclass{siamart1116}

\usepackage[utf8]{inputenc}

\usepackage{notation}

\usepackage{multirow}
\usepackage{booktabs}
\usepackage{adjustbox}

\usepackage{array}
\newcolumntype{L}[1]{>{\raggedright\let\newline\\\arraybackslash\hspace{0pt}}m{#1}}
\newcolumntype{C}[1]{>{\centering\let\newline\\\arraybackslash\hspace{0pt}}m{#1}}
\newcolumntype{R}[1]{>{\raggedleft\let\newline\\\arraybackslash\hspace{0pt}}m{#1}}
\usepackage{graphicx}
\usepackage{tikz}
\usetikzlibrary{shapes}
\usetikzlibrary{arrows}
\usepackage{pgfplots}
\pgfplotsset{compat=newest}

\usepackage[color]{changebar}

\newsiamthm{assumption}{Assumption}
\newsiamthm{remark}{Remark}
\numberwithin{theorem}{section}

\newcommand{\TheTitle}{Deciding Robust Feasibility and Infeasibility Using a Set Containment Approach:\texorpdfstring{\\}{} An Application to Stationary Passive Gas Network Operations}
\newcommand{\TheRunningTitle}{Deciding Robust (In-)Feasibility Using Set Containment}

\newcommand{\TheAuthors}{D. A{\ss}mann, F. Liers, M. Stingl and J. C. Vera}

\headers{\TheRunningTitle}{\TheAuthors}

\title{{\TheTitle}\thanks{
\funding{This research was supported by the DFG within Project~B06 in CRC TRR~154, the Energie Campus Nürnberg (supported by funding of the Bavarian State Government), an STSM Grant from COST Action TD1207, and the ZISC.}}}
\author{
  Denis Aßmann\thanks{Department Mathematik, Friedrich-Alexander-Universit{\"a}t
Erlangen-N{\"u}rnberg, Cauerstra{\ss}e 11, 91058 Erlangen, Germany
    (\email{denis.assmann@fau.de}, \email{frauke.liers@fau.de}, \email{michael.stingl@fau.de}).}
  \and
  Frauke Liers\footnotemark[2]
  \and
  Michael Stingl\footnotemark[2]
  \and
  Juan C. Vera\thanks{Department of Econometrics and Operations Research, Tilburg University, 5000 LE Tilburg, The Netherlands (\email{j.c.veralizcano@uvt.nl}).}
}

\ifpdf
\hypersetup{
  pdftitle={\TheTitle},
  pdfauthor={\TheAuthors}
}
\fi

\usepackage{cleveref}

\begin{document}

\maketitle

% REQUIRED
\begin{abstract}
  In this paper we study feasibility and infeasibility of nonlinear two-stage fully adjustable robust feasibility problems with an empty first stage.
  This is equivalent to deciding whether the uncertainty set is contained within the projection of the feasible region onto the uncertainty-space.
  Moreover, the considered sets are assumed to be described by polynomials.
  For answering this question, two very general approaches using methods from polynomial optimization are presented --- one for showing feasibility and one for showing infeasibility.
  The developed methods are approximated through sum of squares polynomials and solved using semidefinite programs.\\
  Deciding robust feasibility and infeasibility is important for gas network operations, which is a \nonconvex feasibility problem where the feasible set is described by a composition of polynomials with the absolute value function.
  Concerning the gas network problem, different topologies are considered.
  It is shown that a tree structured network can be decided exactly using linear programming.
  Furthermore, a method is presented to reduce a tree network with one additional arc to a single cycle network.
  In this case, the problem can be decided by eliminating the absolute value functions and solving the resulting linearly many polynomial optimization problems. \\
  Lastly, the effectivity of the methods is tested on a variety of small cyclic networks.
  It turns out that for instances where robust feasibility or infeasibility can be decided successfully, level~2 or level~3 of the Lasserre relaxation hierarchy typically is sufficient.
\end{abstract}

% REQUIRED

% REQUIRED

\section{Introduction}
In this paper we study feasibility and infeasibility of nonlinear two-stage fully adjustable robust feasibility problems with an empty first stage.
We further assume that the considered sets, \ie the uncertainty set and the set of feasible solutions, are described by polynomials.
The overall goal of the considered uncertain problem is to answer the question whether for all possible realizations of the data $u \in \uncertaintySet \subseteq \reals^{n_1}$, there is always a solution $x(u) \in \reals^{n_2}$ ($n_1, n_2 \in \naturals$).
If this question can be answered positively, we call the problem ``robust feasible'' and ``robust infeasible''  otherwise.
% Problems of this kind can also be interpreted as fully adjustable two-stage robust optimization problems with an empty first stage.
Given some vector-valued polynomial constraint functions $\genEqFunc\colon\reals^{n_1}\times\reals^{n_2} \rightarrow \reals^{k_1}$ and $\genIneqFunc\colon\reals^{n_1}\times \reals^{n_2} \rightarrow \reals^{k_2}$, the feasibility question can be formulated as
\begin{equation}\label{eq:fully-adj-formulation}
    \forall u \in \uncertaintySet\; \exists x \in \reals^{n_2} \text{ such that } \genEqFunc(u, x) = 0, \; \genIneqFunc(u, x) \geq 0.
\end{equation}
This question can be answered by checking  whether
\begin{equation*}\label{intro:projection-idea}
    \uncertaintySet \subseteq \proj_u(\condset{(u, x)}{\genEqFunc(u, x)=0,\; \genIneqFunc(u, x) \geq 0}).
\end{equation*}
Since set containment implies that each value of $u \in \uncertaintySet$ is associated with at least one feasible solution $x(u)$, the expression in \cref{eq:fully-adj-formulation} holds.
Compared to set containment problems where the constraint-wise description of each set is known, the description of the projection is typically not available or too expensive to compute.
We address this additional challenge in our methods without an explicit construction of the projection.
Solving this type of problem is a first step towards more complex two-stage robust optimization tasks with non-empty first stage and polynomial second stage.
Due to the polynomial structure of the uncertain problem, this leads to polynomial optimization problems.
There are several approaches in the literature on how to construct relaxations of general polynomial problems \cite{Lasserre2001,Nesterov2000,Parrilo2003,Shor1987}.
Given a concrete instance, we use the well known Lasserre SDP relaxation hierarchy \cite{Lasserre2001,Parrilo2003} and solve the resulting semidefinite problems to global optimality.
Since the feasibility question is tackled using a relaxation approach, the constructed problems can't generally give reliable answers for both feasibility and infeasibility at the same time.
For example, due to the relaxation a problem might give a negative answer although the problem is in fact feasible and vice versa.
This makes it necessary to develop two approaches: one for deciding feasibility and one for deciding infeasibility.

Our contribution towards solving this problem is twofold: \\
First, \cref{lemma:feas-method-remove-proj} shows how the problem can still be solved even if an algebraic description of the projected set is not known.
This leads to a series of polynomial optimization problems which can be solved approximately using the Lasserre SDP relaxation hierarchy \cite{Lasserre2001}.
We call this the \emph{feasibility approach}.\\
Second, we develop another polynomial optimization problem to certify infeasibility of the set containment problem in \cref{lemma:infeas-method-remove-proj}.
Again, this so called \emph{infeasibility approach} works without the algebraic description the projected set.

Our methods are inspired by a gas network problem under uncertainty which is a \nonconvex feasibility problem where the feasible set is described by a composition of polynomials with the absolute value function.
The studied gas network problem can be interpreted as a linear network flow problem with additional variables modelling the nodal pressure and constraints linking the pressure difference of two adjacent nodes with the flow over the connecting arc.
For an overview on the problems arising in the operation of gas networks, the reader is referred to \cite{Rios-Mercado2015}.
A special property of the considered type of gas network problems is that the feasible flows are uniquely determined by a piecewise polynomial equation system.
As a consequence, any fixed uncertainty $u$ leads to a unique flow solution $x(u)$ of the problem (which might still be violated by the constraints).
Exploting this property enables us to circumvent an explicit construction of the projected set.

The methods we develop can be applied to two-stage nonlinear robust optimization problems with an empty first stage and polynomial second stage.
In the case of gas network operation, possible first stage variables can model the decisions of the network operator concerning for example the state of the gas compressor machines.
For deciding robust feasibility, we additionally assume that the solutions $x$ depend uniquely on the uncertain data $u$.
For several reasons, the application of standard robust optimization techniques is difficult in this case.
While there are some concepts for solving nonlinear robust optimization tasks \cite{Ben-Tal2015}, they typically require convex nonlinear functions for an exact tractable reformulation.
The canonical way to deal with second stage (``adjustable'') variables is by replacing them with a decision rule of predetermined structure \cite{Ben-Tal2004} which may result in conservative solutions.
If the problem has random recourse, \ie there are products of uncertain parameters and second stage variables, even the introduction of simple affine linear decision rules can only be done  approximately \cite{Ben-Tal2004}.
Another challenge is posed by the assumption that solutions $x$ depend uniquely on the uncertain data.
Thus, this functions $x(u)$ is the only feasible decision rule for the second stage variables.
We therefore use the projection idea to avoid constructing the correct decision rule explicitly.
Regarding the general computation complexity of set containment problems with convex sets, we refer to \cite{Gritzmann1994}.
A more practical treatment for polyhedra and special convex sets is given by \cite{Mangasarian2002}.
Furthermore, a treatment of set containment regarding polytopes and spectahedra can be found in \cite{Kellner2013}.
This work is further extended in \cite{Kellner2015} to encompass projections of polytopes and spectahedra.

Concerning the problem of set containment between basic semialgebraic sets, the general purpose doubly-exponential cylindrical algebraic decomposition algorithm \cite{Collins1975} can be used to eliminate quantifiers from polynomial systems.
It therefore could be used for the combination of projection and set containment.

The framework~\cite{Magron2015} for approximating image sets of compact semialgebraic sets under a polynomial map can also be used to find outer approximations of projected sets.
However, the robust question cannot be decided with their method as an outer approximation of the projected set in~\eqref{intro:projection-idea} could lead to a false positive conclusion regarding robust feasibility.
On the other hand, outer approximations can be used for deciding robust infeasibility.
However, then for each approximation a certificate against set containment still has to be derived.
This would result in an algorithm with two nested optimization tasks, where each task is solved via a sum of squares based hierarchies.
In this respect, our approach seems to be more direct; in particular one of our key contributions is to avoid using an explicit description of the projection.

% Their method would lead to an outer approximations of the projected set in~\eqref{intro:projection-idea} and thus more points would be identified as robust feasible as
% and thus it cannot be used for robust optimization since outer approximations of the projection in \eqref{intro:projection-idea} would lead to false positives.

Optimal control is another field where the problem of set containment of basic semialgebraic sets occurs.
It can be treated through relaxations of the real Positivstellensatz \cite{Jarvis-Wloszek2003}.
This approach is in some sense similar to the techniques in this paper but cannot be applied to the projected problem.

This work is structured as follows:
In \cref{section:problem-setting}, a general introduction to the problem setting is given.
In order to showcase the issue at hand and the solution ideas of this work, a linear network flow problem under uncertainty is presented in \cref{subsec:problem-setting-tree}.
Next, our solution approaches for the set containment problem are laid out in \cref{sec:robust-problems}.
Our main ideas, the infeasibility and feasibility approach for deciding set containment, are presented in \cref{subsec:infeas-approach} and \cref{subsec:feas-model}, respectively.
\Cref{sec:deciding-robustness-of-gas} shows a practical application of the developed methods to the uncertain gas transport problem.
The problem together with some important results concerning the nominal case  are presented in \cref{subsec:passive-gas-net}.
Next, the robust problem is solved for the special case of tree structured networks in \cref{subsec:robust-problems-tree}.
The application part concludes in \cref{subsec:eliminate-abs} with a list of techniques to remove  absolute value functions arising in the gas context.
After converting the problem to a purely polynomial formulation, the developed set containment methods can be applied.
The practical feasibility of the ideas is demonstrated in \cref{sec:numerical-experiments} through a series of numerical results using a number of small network problems.
This work closes with a summary in \cref{sec:conclusion}.

\section{Problem Description and the Setting Studied Here}\label{section:problem-setting}
A class of parameterized nonlinear feasibility problems is studied.
Let $\genEqFunc\colon\reals^{n_1}\times\reals^{n_2} \rightarrow \reals^{k_1}$ and $\genIneqFunc\colon\reals^{n_1}\times \reals^{n_2} \rightarrow \reals^{k_2}$ ($n_1$, $n_2$, $k_1$, $k_2 \in \naturals$) be some (possibly nonlinear) constraint functions.
For the solution approaches presented in this paper, these functions are assumed to be polynomial.
The first argument of each function is assumed to be a parameter $u$ which is shared by both $\genEqFunc$ and $\genIneqFunc$.
This parameter is often called the \emph{uncertainty} or \emph{uncertain data} of the problem which is an element of an priori given \emph{uncertainty set} $\uncertaintySet \subseteq \reals^{n_1}$.
Our goal is to answer the question whether for all possible realizations $u \in \uncertaintySet$ there is always a feasible solution $x \in \reals^{n_2}$ of the problem:
\begin{equation}\label{eq:fully-adj-formulation-setting}
    \forall u \in \uncertaintySet\; \exists x \in \reals^{n_2} \text{ such that } \genEqFunc(u, x) = 0, \; \genIneqFunc(u, x) \geq 0.
\end{equation}
Let $\feasCoeffsAndFlows = \condset{(u, x) \in \reals^{n_1}\times\reals^{n_2}}{\genEqFunc(u, x) = 0, \; \genIneqFunc(u, x) \geq 0}$ be the set of all feasible pairs of uncertain data $u$ and problem solution $x$.
Question \cref{eq:fully-adj-formulation-setting} can be answered by checking whether the set containment condition
\begin{equation}\label{u-proj-b}
    \uncertaintySet \subseteq \proj_u(\feasCoeffsAndFlows)
\end{equation}
holds.
Since set containment implies that each value of $u \in \uncertaintySet$ is associated with at least one feasible solution $x(u)$, the expression in \eqref{eq:fully-adj-formulation-setting} is satisfied.

In the next subsection, the set containment idea is further explored on the example of a simple linear network flow problem over a tree.

\subsection{Introductory Example: Linear Flow Problem over a Tree}\label{subsec:problem-setting-tree}
We want to further illustrate the problem and its possible solution approaches by means of a simple example.
Let a linear flow problem be given over a tree with lower and upper edge capacities and uncertain demands.
The data appears as an uncertain right hand side of the flow balance equations.
We assume that the demand $u$ of all nodes except some fixed root node fluctuates within a hypercube $\uncertaintySet$.
The model can then be stated as
\begin{equation*}
    \forall u \in \uncertaintySet \; \exists x \colon
    \left\{
    \begin{aligned}
        Ax &= u,\\
        \lb{x} \leq &x \leq \ub{x}
    \end{aligned}
    \right.
\end{equation*}
for some non-singular matrix $A$, see \cref{subsec:passive-gas-net} for details.
After substituting $x = A^{-1}u$, the problem is equivalent to
\begin{equation*}
    \forall u \in \uncertaintySet \colon
        \lb{x} \leq A^{-1}u \leq \ub{x},
\end{equation*}
or
\begin{equation}\label{model:tree-flow-set-containment}
    \uncertaintySet \subseteq \condset{u}{\lb{x} \leq A^{-1}u \leq \ub{x}} = \proj_u\left(\condset{(u, x)}{Ax = u,\; \lb{x} \leq x \leq \ub{x}}\right)
\end{equation}
when stated as a set containment problem.
Since both sets are polyhedral, the question can be decided by optimizing over the remaining constraint functions: if
\begin{align*}
    \max_{u \in \uncertaintySet} (A^{-1}u)_i \leq \ub{x}_i
    \quad\text{and}\quad
    \min_{u \in \uncertaintySet} (A^{-1}u)_i \geq \lb{x}_i
\end{align*}
hold for all $i=1,\ldots,n_2$, so does the set containment condition.
By using linear duality, these inequalities can be checked with one linear optimization problem, see \cref{lemma:polyhedral-set-containmet}.

In this example, we were able to exploit the simple structure to directly construct the projected set in equation~\cref{model:tree-flow-set-containment}.
For more complicated linear or nonlinear constraints, this may not always be possible or computationally too expensive.
For treating the arising problems, we will use ideas from polynomial optimization.

\subsection{Polynomial Optimization}\label{subsec:poly-methods}
We denote with $\naturals = \{1,2,\ldots\}$ the set of natural numbers and with $\naturalsWithZero = \{0,1,2,\ldots\}$ the set of natural numbers including zero.
Let $\reals[x]:=\reals[x_1,\ldots,x_n]$ denote the set of polynomials in $n$ variables with real coefficients.
A polynomial $p \in \reals[x]$ is defined as $p(x) = \sum_{\alpha \in \naturalsWithZero^n} p_\alpha x^{\alpha}$ with coefficients $p_\alpha \in \reals$ and monomials $x^{\alpha} = (x_1^{\alpha_1}, \ldots, x_n^{\alpha_n})$ for $\alpha \in \naturalsWithZero^n$.
With $\abs{x^\alpha} = \sum_{i} \alpha_i$, define the degree of $p$ as $\deg(p):=\max \condset{\abs{x^\alpha}}{p_\alpha \neq 0}$.
Let $\polynomialsPosOnSet[S] = \condset{p\in \reals[x]}{p(x) \geq 0, \,\forall x \in S}$ (resp. $\mathcal{P}  =\polynomialsPosOnSet[\reals^n]$) denote the set of \nonnegative polynomials on a subset $S \subseteq \reals^n$ (resp. on $\reals^n$).

Polynomial optimization is the problem of optimizing a polynomial over a \emph{basic semi-algebraic set} $S = \{x \in \reals^n: q_1(x) \ge 0,\dots,q_m(x) \ge 0\}$. Every polynomial optimization problem can be written as optimizing a linear function over  the cone $\polynomialsPosOnSet[S]$ of \nonnegative polynomials on $S$. Optimizing over $\polynomialsPosOnSet[S]$ is $\mathcal{NP}$-hard for most (interesting) choices of $S$.
Hierarchies of tractable approximations of the cone $\polynomialsPosOnSet[S]$  are typically constructed through sum of squares (SOS) relaxations (\cite{Lasserre2001}), which correspond to semidefinite liftings of subsets of $\polynomialsPosOnSet[S]$ into higher dimensions. The construction is motivated by results related to representations of non-negative polynomials as SOS and the dual theory of moments.  The convergence of Lasserre's method is based on the assumption that $\{q_1,\dots,q_m\}$, the given description of $S$, allows the application of Putinar's Theorem \cite{Putinar1993}. In particular, it assumes $S$ is compact.

To construct Lasserre's hierarchy, first the polynomial optimization problem is reformulated as a conic problem over $\polynomialsPosOnSet[\mathcal{S}]$ using
\[
\inf\{p(x): x \in \mathcal{S}\} = \sup\{\lambda \in \reals: p - \lambda \in \polynomialsPosOnSet[\mathcal{S}]\}.
\]

The \emph{truncated quadratic module} of level $d$ is defined as
\begin{equation}
    M_{d}[\mathcal{S}] = \condSet{\sigma_0(x) + \sum_{i=1}^m \sigma_i(x) q_i(x)}{
    \begin{gathered}
        \sigma_0, \sigma_i \text{ is sum of squares}\\
        \deg(\sigma_0) \leq 2d, \, \deg(\sigma_i q_i) \leq 2d
    \end{gathered}
    }.
\end{equation}
This set can be expressed as the feasible region of $m+1$ semidefinite constraints with linear equalities over the coefficients of $\sigma_0$ and $\sigma_i q_i$ \cite{Shor1987}.
Furthermore, as $M_{d}[\mathcal{S}]\subseteq M_{d+1}[\mathcal{S}] \subseteq \polynomialsPosOnSet[\mathcal{S}]$ holds, this set can be used as an approximation for $\polynomialsPosOnSet[\mathcal{S}]$. Notice that by increasing $d$, a sequence of semidefinite relaxations of increasing size is obtained.

Lasserre shows \cite{Lasserre2001} that under mild conditions, the optimal objective value over these relaxations converges to the optimal value over $\polynomialsPosOnSet[\mathcal{S}]$. \Cref{prop:Lascvgce} states a slightly more general result using our notation.
\begin{proposition}\label{prop:Lascvgce} Let $q_0,q_1,\dots,q_m \in \reals[x]$ be given. Let $\mathcal{S} = \{x \in \reals^n: q_1(x) \ge 0,\dots,q_m(x) \ge 0\}$.
For any pair of polynomials \(q = \sum_\alpha q_\alpha x^\alpha\) and \(p = \sum_\alpha q_\alpha x^\alpha\)  we define \(q \circ p = \sum_\alpha q_\alpha p_\alpha\).
Consider the optimization problem $ \mu = \sup \{q_0 \circ p: p \in   \polynomialsPosOnSet[\mathcal{S}]\}$ and the sequence of relaxations $ \mu_d = \sup \{q_0 \circ p: p \in   M_d[\mathcal{S}]\}$.
Assume there exists a real-valued polynomial $u(x) =  \sigma_0 + \sum_{i=1}^m q_i \sigma_i$ where $\sigma_i$ is SOS for all $i$ and such that $\{x:u(x) \geq  0\}$ is compact.

Then (Putinar \cite{Putinar1993})
\begin{align*}
  &M_1[\mathcal{S}] \subseteq M_2[\mathcal{S}]  \subseteq \cdots \subseteq M_d[\mathcal{S}] \subseteq \cdots \subseteq  \polynomialsPosOnSet[\mathcal{S}] \\
  \text{ and } &\{p \in \reals[x]:p(s) > 0 \: \forall s \in \mathcal{S}\} \subseteq \bigcup_{d>0} M_d[\mathcal{S}]
\end{align*}
and therefore (using same ideas as in Lasserre \cite{Lasserre2001})
\begin{equation*}
   \mu^{1} \le \mu^{2}  \le \cdots \le \mu^{d} \le \cdots \le \mu \mbox{ and } \mu^r \rightarrow \mu \text{ as } d \rightarrow \infty.
\end{equation*}
\end{proposition}
In other words, using Lasserre's hierarchy for general polynomial optimization problems one may approximate the global optimal value $\mu$ as closely as desired by solving a sequence of semidefinite problems with increasing size of the semidefinite matrices and number of constraints.

\section{Deciding Robust Feasibility and Infeasibility for the General Case}\label{sec:robust-problems}
In this section, the two approaches for deciding robustness are developed.
We present a method for certifying infeasibility in \cref{subsec:infeas-approach} as well as a method for proving feasibility in \cref{subsec:feas-model}.

\subsection{A Set Containment Approach for Certifying Infeasibility}\label{subsec:infeas-approach}
A robust optimization problem is said to be infeasible if a scenario $\fixed{u} \in \uncertaintySet$ exists whose corresponding problem is infeasible.
We first introduce an abstract model involving arbitrary functions for solving this problem.
The model is then adapted to the considered case of polynomial functions.
With this approach, negative certificates for set containment of two basic semi-algebraic sets can be found.
Recall that a set $\mathcal{S}$ is called \emph{basic semi-algebraic}, if it is of the form
\begin{equation*}
    \mathcal{S} = \condSet{x}{p_i(x) \geq 0, \quad i = 1,\ldots,n}.
\end{equation*}
where $p_i(x) \in \reals[x]$ for $i = 1,\ldots,n$ ($n \in \naturals$) are polynomials.
% In conjunction with a projection argument, this idea is used later on to certify infeasibility of the gas network problem.
For any set $\mathcal{S}$, let $\funcsPosOnSet[\mathcal{S}] := \condSet{f\colon \reals^n \rightarrow \reals}{f(x) \geq 0 \text{ for } x \in \mathcal{S}}$ be the set of all nonnegative functions on $\mathcal{S}$.
The set $\funcsPosOnSet[\mathcal{S}]$ is nonempty since it always contains $h(x) \equiv 0$, regardless of the particular choice of $\mathcal{S}$.

Let $\subSet, \superSet$ be any subsets of $\reals^n$.
It is clear that
\begin{equation}\label{eq:false-set-containement-char}
    \subSet \not\subseteq \superSet \iff \exists x \in \subSet\colon x \not\in \superSet \iff \subSet \minus \superSet \neq \emptyset,
\end{equation}
where we denote with \(\subSet \minus \superSet = \condSet{x \in \subSet}{x \not\in \superSet}\) the set difference of \(\subSet\) and \(\superSet\).
With this definition, \cref{eq:false-set-containement-char} can be extended to
\begin{equation*}
    \subSet \minus \superSet \neq \emptyset \iff \exists\,f \in \funcsPosOnSet[\superSet] \text{ and } x \in \subSet \text{ such that } f(x) < 0.
\end{equation*}
The last expression can be rewritten using an optimization problem.
Let the abstract separation problem \cref{prob:abstract-separation} be defined as
\begin{equation}\tag{ASep}\label{prob:abstract-separation}
    \begin{aligned}
        \inf & \,f(x), \\
        x &\in \subSet, \\
        f &\in \funcsPosOnSet[\superSet].
    \end{aligned}
\end{equation}
We employ the usual definition of $\inf_X f(x) = +\infty$ if $X=\emptyset$.
For the optimal value of \cref{prob:abstract-separation} it holds that
\begin{equation*}
    \inf_{x \in \subSet, f \in \funcsPosOnSet[\superSet]} f(x) =
    \begin{cases}
        +\infty, &\text{ if } \subSet = \emptyset \\
        0,       &\text{ if } \subSet \neq \emptyset \text{ and } \subSet \subseteq \superSet, \\
        -\infty, &\text{ if } \subSet \neq \emptyset \text{ and } \subSet \not\subseteq \superSet.
    \end{cases}
\end{equation*}
Combining the first two cases yields
\begin{equation}
    \subSet \not\subseteq \superSet \iff \inf_{x \in \subSet, f \in \funcsPosOnSet[\superSet]} f(x) = -\infty.
\end{equation}

In order to tackle this optimization task in practice, the abstract problem is approximated by a polynomial optimization problem.
We first replace the set of functions $\funcsPosOnSet[\superSet]$ by the set
\begin{equation*}
    \polynomialsPosOnSet[\superSet] := \condSet{p \in \polynomials}{p(x) \geq 0 \text{ for } x \in \superSet}
\end{equation*}
of polynomials that are nonnegative on \superSet.
Since both $p$ and $x$ are variables, $p(x)$ cannot be cast directly as part of a polynomial optimization problem.
Therefore, instead of minimizing $p(x)$, we minimize the Lebesgue integral of $p$ over $\subSet$.
A negative integral \(\int_{\subSet} p(x) \d \mu\) implies the existence of some \(\fixed{x} \in \subSet\) with \(p(\fixed{x}) < 0\):
\begin{equation*}
    \begin{aligned}
        \inf_p & \, \int_{\subSet} p(x) \d \mu, \\
        p &\in \polynomialsPosOnSet[\superSet].
    \end{aligned}
\end{equation*}

Using the definition $p(x) = \sum_\alpha p_\alpha x^\alpha$, the objective can be rewritten in terms of the moments of $\mu$:
\begin{equation}\tag{PolySep}\label{prob:poly-sep}
  \begin{gathered}
    \inf_p \int_{\subSet} p \d\mu = \inf_p \sum_\alpha p_\alpha \int_{\subSet} x^\alpha \d\mu\\
    p \in \polynomialsPosOnSet[\superSet].
  \end{gathered}
\end{equation}
Since the moments $\int_{\subSet} x^\alpha \d\mu$ can be calculated in advance, the objective of \cref{prob:poly-sep} is a linear function in $p$.

We call this problem the polynomial separation problem.
If there exists $p$, such that the integral over $\subSet$ is negative, there must be some point $x \in \subSet$ with $p(x) < 0$.
Then, by definition of $p$, it holds that $x \not\in \superSet$.

The integration is a weaker test for the existence of an $x \in \subSet$ with $p(x) < 0$ than just evaluating $p(x)$ (see Lemma~\ref{lemma:pos-polynomial-on-open-ball}).
For practical applications, the moments $\int_{\subSet} x^\alpha \d \mu$ need to available.
With respect to the presented robust gas network problem, this is no limitation since $\subSet = \uncertaintySet$ is a hypercube.
In a similar context, precomputed moments of a simple superset, \eg of a sphere or a box, are used to approximate the volume of an arbitrary basic compact semialgebraic set in \cite{Henrion2009b}.

The next lemma identifies conditions for \subSet, \superSet for which a polynomial $p \in \polynomialsPosOnSet[\superSet]$ exists with $\int_{\subSet} p(x) \d \mu < 0$.
This means that under these conditions, problems \cref{prob:abstract-separation,prob:poly-sep} are equivalent.

\begin{lemma}\label{lemma:pos-polynomial-on-open-ball}
    Let $\subSet, \superSet \subseteq \reals^n$ be two bounded sets with $\subSet \minus \superSet \neq \emptyset$.
    Suppose that $\subSet \minus \superSet$ contains an open subset. \\
    Then there exists a polynomial $p \in \polynomialsPosOnSet[\superSet]$ with $\int_{\subSet} p(x) \d \mu < 0$.
\end{lemma}
\begin{proof}
    Since $\subSet \minus \superSet$ contains an open subset, there exists $x_0 \in \reals^n$ and $r > 0$ such that $\subSet \minus \superSet \supseteq \openBall{r}{x_0} =: \condSet{x \in \reals^n}{\norm{x - x_0} < r}$.
    Without loss of generality, we assume that $x_0 = 0$.
    This can always be guaranteed by applying a simple translation to \subSet and \superSet.
    Due to both sets being bounded, there exists an $R > r$ such that $\superSet, \subSet \subseteq \openBall{R}{0}$.

    We prove this lemma by constructing a polynomial $p\colon \reals^n \rightarrow \reals$ that is non-negative on $\openBall{R}{0} \minus \openBall{r}{0} \supseteq \superSet$ and satisfies $\int_{\openBall{R}{0}} p \d \mu < 0$.
    If such a $p$ exists, it holds that
    \begin{align*}
        \int_{\subSet} p \d \mu &= \int_{\subSet \minus \openBall{r}{0}}p \d \mu + \int_{\openBall{r}{0}}p \d \mu \\
        &\leq \int_{\openBall{R}{0} \minus \openBall{r}{0}}p \d \mu + \int_{\openBall{r}{0}}p \d \mu = \int_{\openBall{R}{0}}p \d \mu < 0.
    \end{align*}

    In order to construct $p$, let
    \begin{equation*}
        q(t) := [c_1 (t - c_2)]^2
    \end{equation*}
    be a univariate polynomial with constants $c_1:=\frac{2}{R^2-r^2}$, $c_2:=\frac{R^2+r^2}{2}$.
    By construction, the following holds:
    \begin{subequations}\label{eq:qpoly-props}
        \begin{align}
            q(c_2) = 0&                                           \label{eq:qpoly-props:a}\\
            q(t^2) = 1           &\text{ iff }  t \in \{r, R\},  \label{eq:qpoly-props:b}\\
            q(t^2) \geq 1        &\text{ for } t \in [0, r],      \label{eq:qpoly-props:c}\\
            0 \leq q(t^2) \leq 1 &\text{ for } t \in [r, R].      \label{eq:qpoly-props:d}
        \end{align}
    \end{subequations}
    Taking the $l$-th ($l \in \naturals$) power of $q$ preserves properties \cref{eq:qpoly-props:a,eq:qpoly-props:b,eq:qpoly-props:c,eq:qpoly-props:d}.
    Furthermore, the polynomial
    \begin{equation*}
        p_l(t) := 1-q^l(t)
    \end{equation*}
    satisfies
    \begin{align*}
        p_l(c_2) = 1&\\
        p_l(t^2) = 0           &\text{ iff }  t \in \{r, R\},\\
        p_l(t^2) \leq 0        &\text{ for } t \in [0, r],\\
        0 \leq p_l(t^2) \leq 1 &\text{ for } t \in [r, R].
    \end{align*}

    We now show that there exists $l \in \naturals$ such that the radial symmetric polynomial $p_l(\norm{x}^2)$ is non-negative on $\openBall{R}{0} \minus \openBall{r}{0} \supseteq \superSet$ and satisfies $\int_{\openBall{R}{0}} p_l(\norm{x}^2) \d \mu < 0$:
    \begin{align*}
        \int_{\openBall{R}{0}} p_l(\norm{x}^2) \d\mu &= \int_{\openBall{R}{0} \minus \openBall{r}{0}} p_l(\norm{x}^2) \d\mu + \int_{\openBall{r}{0}} p_l(\norm{x}^2) \d\mu \\
        &\leq \int_{\openBall{R}{0} \minus \openBall{r}{0}} 1 \d\mu + \int_{\openBall{r}{0}} 1 - q^l(\norm{x}^2) \d\mu \\
        &= \int_{\openBall{R}{0}} 1 \d\mu - \int_{\openBall{r}{0}} q^l(\norm{x}^2)\d\mu.
    \end{align*}
    In order to complete the proof, we show that $\lim_{l \rightarrow \infty} \int_{\openBall{r}{0}} q^l(\norm{x}^2)\d\mu = \infty$.
    Using a substitution of variables and exploiting the radial symmetry, the integral over the $n$-dimensional ball can be transformed to a univariate integral:
    \begin{equation*}
        \int_{\openBall{r}{0}} q^l(\norm{x}^2)\d\mu = \overbrace{n \int_{\openBall{1}{0}} 1 \d\mu}^{=:\Gamma > 0} \int_0^r q^l(t^2) t^{n-1} \d t
    \end{equation*}
    Now we calculate the difference between two integrals in the sequence while omitting the positive coefficient $\Gamma$:
    \begin{align*}
        &\int_0^r q^{l+1}(t^2) t^{n-1} \d t - \int_0^r q^l(t^2) t^{n-1} \d t \\
        =& \int_0^r \overbrace{q^l(t^2)}^{\geq 1}\overbrace{t^{n-1}}^{\geq 0}\overbrace{\left(q(t^2) - 1\right)}^{\geq 0} \d t\\
        \geq& \int_0^r t^{n-1} \left(q(t^2) - 1\right) \d t = c > 0
    \end{align*}
    Since the difference between two consecutive elements of the series is bounded from below by a strictly positive constant $c$, the series diverges to $+\infty$.
    This implies the existence of some $l \in \naturals$ such that $\int_{\openBall{R}{0}} p_l(\norm{x}^2) \d\mu < 0$.
\end{proof}

Using $p(u) = \sum_{\alpha} p_\alpha u^\alpha$, the corresponding optimization problem to certify infeasibility of the robust problem is
\begin{equation}\tag{PolySepProj}\label{prob:poly-sep-proj}
\begin{gathered}
    \inf_p \sum_\alpha p_\alpha \int_{\uncertaintySet} u^\alpha \d\mu,\\
    p \in \polynomialsPosOnSet[\proj_u(\feasCoeffsAndFlows)].
\end{gathered}
\end{equation}
Without explicit knowledge of the projection $\proj_u(\feasCoeffsAndFlows)$, it is unclear how the set $\polynomialsPosOnSet[\proj_u(\feasCoeffsAndFlows)]$ can be expressed as part of a polynomial optimization problem.
We present an equivalent model which expresses this constraint by introduction of additional linear constraints over the coefficients of the unknown polynomial.
\begin{lemma}\label{lemma:infeas-method-remove-proj}
    Consider the two optimization problems
    \begin{align*}(1)\quad
        \begin{gathered}
            \inf_p \sum_\alpha p_\alpha \int_{\uncertaintySet} u^\alpha \d\mu,\\
            p \in \polynomialsPosOnSet[\proj_u(\feasCoeffsAndFlows)],
        \end{gathered}&&\text{and}&&(2)\quad
        \begin{gathered}
            \inf_{\tilde{p}} \sum_{\alpha} \tilde{p}_{\alpha,\beta} \int_{\uncertaintySet} u^\alpha x^\beta \d\mu,\\
            \tilde{p}_{\alpha, \beta} = 0 \quad \forall \beta \neq 0,\\
            \tilde{p} \in \polynomialsPosOnSet[\feasCoeffsAndFlows],
        \end{gathered}
    \end{align*}
    where $p(u) = \sum_{\alpha} p_{\alpha} u^\alpha$ is a polynomial in \(u\) and $\tilde{p}(u, x) = \sum_{\alpha, \beta} \tilde{p}_{\alpha, \beta} u^\alpha x^\beta$ is a polynomial in both \(u\) and \(x\).

    Any feasible point $\optimal{p}$ of (1) can be extended to a feasible point $\optimal{\tilde{p}}$ of (2) and vice versa.
    Furthermore, the feasible points $\optimal{p}$ and $\optimal{\tilde{p}}$ have the same objective values.
\end{lemma}
\begin{proof}
    ``$\Rightarrow$'':
    Let $\optimal{p}$ be any feasible point of (1) with objective value $\optimal{z} = \sum_\alpha \optimal{p}_\alpha \int_{\uncertaintySet} u^\alpha \d\mu$.
    Consider the inclusion map from  $\reals[u]$ to $\reals[u,x]$, which maps $\optimal{p}$ to $\optimal{\tilde{p}}$ where $\optimal{\tilde{p}}(u, x) = \sum_{\alpha, \beta} \optimal{\tilde{p}}_{\alpha, \beta} u^\alpha x^\beta$ where
    \begin{equation}
        \optimal{\tilde{p}}_{\alpha, \beta} :=
        \begin{cases}
            \optimal{p}_{\alpha}, &\text{ if } \beta = 0,\\
            0,          &\text{ if } \beta \neq 0.
        \end{cases}
    \end{equation}
    By construction, for any $u \in \proj_u(\feasCoeffsAndFlows)$ and $x \in \reals^{|N|}$, we have $\optimal{\tilde{p}}(u,x) = \optimal{p}(u) \ge 0$. Therefore  $\optimal{\tilde{p}} \in \polynomialsPosOnSet[\proj_u(\feasCoeffsAndFlows) \times \reals^{|N|}] \subseteq \polynomialsPosOnSet[\feasCoeffsAndFlows]$.
    That is $\optimal{\tilde{p}}$ is feasible for (2).\\
    ``$\Leftarrow$'':
    Let $\optimal{\tilde{p}}$ be any feasible point of $(2)$.
    Since all coefficients $\optimal{\tilde{p}}_{\alpha, \beta}$ with $\beta\neq0$ are zero, $\optimal{\tilde{p}}$ is independent of $x$ and it holds that $\optimal{\tilde{p}} \in \polynomialsPosOnSet[\proj_u(\feasCoeffsAndFlows) \times \reals^{|N|}]$.
    Let $\optimal{p}(u) = \sum_{\alpha} \optimal{p}_{\alpha} u^\alpha$ be the remaining polynomial in $u$.
    % Furthermore, $\optimal{\tilde{p}}$ easily maps to a polynomial $\optimal{p} = \sum_{\alpha} \optimal{p}_{\alpha} u^\alpha$.
    Together with $\optimal{\tilde{p}} \in \polynomialsPosOnSet[\proj_u(\feasCoeffsAndFlows) \times \reals^{|N|}]$, this implies $\optimal{p} \in \polynomialsPosOnSet[\proj_u(\feasCoeffsAndFlows)]$.
    % For the objective values, the same calculation as in the last paragraph can be applied:
\end{proof}

For the remainder of this section, we assume that the problem is robust infeasible, i.e. $\setDiff:= \uncertaintySet \setminus \proj_u(\feasCoeffsAndFlows)$ is non-empty.
In order to apply \cref{lemma:pos-polynomial-on-open-ball}, $\setDiff$ has to contain an open subset.
The next proposition shows that for the given sets, this is no restriction since such a subset always exists.
Given a set $\mathcal{S} \subseteq \reals^n$, we denote with $\closure(\mathcal{S})$, $\interior(\mathcal{S})$, $\boundary\mathcal{S}$, and $\setComplement{\mathcal{S}}$ the closure, interior, boundary, and complement of $\mathcal{S}$, respectively.
For this paper, the uncertainty set $\uncertaintySet$ is assumed to be a full-dimensional hypercube or full-dimensional polyhedron.
Therefore, $\uncertaintySet = \closure(\interior(\uncertaintySet))$ always holds for our choices of $\uncertaintySet$.
\begin{proposition}\label{prop:open-subset-always-exists}
    Let $\uncertaintySet \subseteq \reals^{n_1}$ be a set with $\uncertaintySet = \closure(\interior(\uncertaintySet))$. Let $\feasCoeffsAndFlows \subseteq \reals^{n_1} \times \reals^{n_2}$ be a compact set and let  $\setDiff= \uncertaintySet \setminus \proj_u(\feasCoeffsAndFlows) \neq \emptyset$.
    Then $\setDiff$ contains an open subset.
\end{proposition}
\begin{proof}
    We need to show that $\interior(\setDiff) = \interior(\uncertaintySet) \cap \setComplement{(\proj_u(\feasCoeffsAndFlows))} \neq \emptyset$.
    Since $\feasCoeffsAndFlows$ is compact, $\proj_u(\feasCoeffsAndFlows)$ is closed and thus $\setComplement{(\proj_u(\feasCoeffsAndFlows))}$ is an open set.

    Pick any $x \in \setDiff = \uncertaintySet \cap \setComplement{(\proj_u(\feasCoeffsAndFlows))}$.
    If $x \in \interior(\uncertaintySet)$, then $x \in \interior({\setDiff})$ holds as well since $\setComplement{(\proj_u(\feasCoeffsAndFlows))}$ is an open set.\\
    Otherwise, assume that $x \in \partial \uncertaintySet$.
    With $x \in \setComplement{(\proj_u(\feasCoeffsAndFlows))}$, there exists $\varepsilon > 0$ such that $\openBall{\varepsilon}{x} \subseteq \setComplement{(\proj_u(\feasCoeffsAndFlows))}$.
    Since $\uncertaintySet = \closure(\interior(\uncertaintySet))$, there exists $y \in \interior(\uncertaintySet) \cap \openBall{\varepsilon}{x} \subseteq \setComplement{(\proj_u(\feasCoeffsAndFlows))}$.
    Therefore, $y \in \interior(\setDiff)$.

    % Since $\uncertaintySet$ is a full-dimensional hypercube, it holds that $\uncertaintySet = \closure(\interior(U))$ where $\closure(\cdot)$ and $\interior(\cdot)$ denote the closure and interior, respectively.
    % for all $u\in \partial\uncertaintySet$ and $\varepsilon > 0$ exists some $u' \in \uncertaintySet^\mathrm{o}$ with $\norm{u - u'} < \varepsilon$.\\
    % Therefore, $\setDiff = \uncertaintySet \setminus \proj_u(\feasCoeffsAndFlows) = \uncertaintySet \cap (\proj_u(\feasCoeffsAndFlows))^\mathrm{C}$.
    % Pick any $x \in \setDiff$.
    % By definition, $x$ is an element of the open set $(\proj_u(\feasCoeffsAndFlows))^\mathrm{C}$.
    % If $x$ is an element of the open set $\uncertaintySet^\mathrm{o}$ as well, it is clear that there exists an open subset of $\setDiff$ around $x$.
    % Otherwise, if $x \in \partial\uncertaintySet$, there is an arbitrarily close point $x' \in \uncertaintySet^\mathrm{o}$.
    % Since $x \in (\proj_u(\feasCoeffsAndFlows))^\mathrm{C}$ and $(\proj_u(\feasCoeffsAndFlows))^\mathrm{C}$ is an open set, $x'$ and $\varepsilon' > 0$ can be found such that $\openBall{\varepsilon'}{x'} \subseteq \setDiff$.
    This concludes that for the given sets, $\setDiff$ always contains an open subset if $\setDiff$ is non-empty.
\end{proof}

With \cref{prop:open-subset-always-exists} and \cref{lemma:pos-polynomial-on-open-ball}, the  separation problem \cref{prob:poly-sep} can certify infeasibility if the assumptions of \cref{prop:open-subset-always-exists} are satisfied.
In practice, this optimization problem is then approximated by some finite relaxation of the Lasserre hierarchy using \cref{prop:Lascvgce}.
The question remains whether for sufficiently large levels of the hierarchy, the separation polynomial as given by \cref{lemma:pos-polynomial-on-open-ball} can always be found.
After all, not all positive polynomials can be expressed by sum of square polynomials.
This is no restriction as the following proposition shows:
\begin{proposition}
    There is some finite level of the Lasserre hierarchy for which the corresponding SDP approximation of \cref{prob:poly-sep} yields a negative objective if $\setDiff \neq \emptyset$.
\end{proposition}
\begin{proof}
    By \cref{prop:open-subset-always-exists}, $\setDiff \neq \emptyset$ implies the existence of some open subset in $\setDiff$.
    Then \cref{lemma:pos-polynomial-on-open-ball} guarantees the existence of a polynomial $p$ with strictly negative objective value for the abstract polynomial optimization problem. \\
    Consider then the SDP approximation of \cref{prob:poly-sep}.
    Since SOS-polynomials are dense (see \cite{Lasserre2007}) in the set of non-negative polynomials and by the continuity of the integral, there is always a SOS-polynomial close to the $p$ with a negative objective value.
\end{proof}

\subsection{A Set Containment Approach for Certifying Feasibility}\label{subsec:feas-model}
In general, deciding robust feasibility is equivalent to answering the set containment question
\begin{equation*}
    \uncertaintySet \subseteq \proj_u(\feasCoeffsAndFlows).
\end{equation*}
Since an explicit description of $\proj_u(\feasCoeffsAndFlows)$ is typically not available, we next show how the question above can be decided equivalently using non-projected sets.
The set $\feasCoeffsAndFlows$ of all feasible pairs of uncertain data \(u \in \reals^{n_1}\) and problem solution \(x\in\reals^{n_2}\) can be written naturally as an intersection in the following way:
\begin{equation*}
  \feasCoeffsAndFlows = \underbrace{\{(u, x) \mid \genEqFunc(u, x) =0 \}}_{=:\mathcal{G}} \cap \underbrace{\{(u, x) \mid \genIneqFunc(u, x) \geq0\}}_{=:\mathcal{H}}.
\end{equation*}

In our approach, we require that for all possible realizations of the uncertain data \(u \in \uncertaintySet\), the equation system \(\genEqFunc(u, x)=0\) has a unique solution in \(x\).
Let \(\mathcal{G}_\uncertaintySet = \condset{(u, x)}{u \in \uncertaintySet, \, \genEqFunc(u, x) = 0}\) be the restriction of \(\mathcal{G}\) to the pairs containing elements of the uncertainty set.
\begin{assumption}\label{asmp:G-is-func-graph}%
  For all \(\hat{u} \in \uncertaintySet\), the system \(\genEqFunc(\hat{u}, x) = 0\) has exactly one solution \(\hat{x} \in \reals^{n_2}\).
\end{assumption}
If the previous assumption is satisfied, let \(\genEqSolFunc\colon\uncertaintySet \rightarrow \reals^{n_2}\) be the (unique) function that maps elements of the uncertainty set to solutions. That is,
for all \(\hat{u} \in \mathcal{U}\), let \(\genEqSolFunc(\hat{u}) \in \reals^{n_2}\) be the unique solution to \(\genEqFunc(\hat{u}, x)=0\).
Using the uncertainty-to-solution function \(\genEqSolFunc\), the set \(\mathcal{G}_\uncertaintySet\) can be rephrased as \(\mathcal{G}_\uncertaintySet = \condset{(u, g(u))}{u \in \uncertaintySet}\).

\begin{remark}
    Uniqueness of solutions as in~\cref{asmp:G-is-func-graph} is a feature of many physical systems that are modeled as a partial differential equation (PDE) system.
    For instance, for a wide class of boundary value problems the uniqueness of the solution follows from the famous Lemma of Lax--Milgram~\cite{Lax1954} for arbitrary right-hand sides using a coercivity assumption.
    This directly implies that uniqueness and~\cref{asmp:G-is-func-graph} also hold for PDEs with uncertain coefficients, as long as the coercivity is maintained on the whole uncertainty set.

    The set \(\mathcal{G}_\uncertaintySet\) comprises all uncertainty-dependent solutions of the state equation \(\genEqFunc(u,x)=0\), whereas the set \(\mathcal{H}\) is described by the given state constraints.
    Moreover, we remark that an explicit construction of the function $g(\cdot)$ is never required; we merely introduce \(g\) to simplify the presentation.
\end{remark}

The next lemma shows how the unique dependency between \(u\) and \(x\) leads to an equivalent projection-less formulation of the set containment question \(\uncertaintySet \subseteq \proj_u(\feasCoeffsAndFlows)\).
\begin{lemma}\label{lemma:feas-method-remove-proj}
    Let $\uncertaintySet \subseteq \reals^{n_1}$ and let $\genEqSolFunc\colon\uncertaintySet \rightarrow \reals^{n_2}$.
    Let $\genIneqFunc_i\colon\reals^{n_1}\times\reals^{n_2} \rightarrow \reals$  ($n_1$, $n_2 \in \naturals$) for $i=1,\ldots,k_2$ be functions.
    Let
    %\begin{align*}
    %    \eqSet_\uncertaintySet &:= \condset{(u, x) \in \reals^{n_1} \times \reals^{n_2}}{u \in \uncertaintySet, \, x = \genEqSolFunc(u)},\\
    %    \inEqSet &:= \condset{(u, x) \in \reals^{n_1} \times \reals^{n_2}}{\genIneqFunc_i(u, x) \geq 0, \, i = 1,\ldots,k_2}
    %\end{align*}
    %with
    \begin{equation*}
        \eqSet_\uncertaintySet := \condset{(u, \genEqSolFunc(u)) \in \reals^{n_1} \times \reals^{n_2}}{u \in \uncertaintySet} \text{ and }
        \inEqSet \subseteq \reals^{n_1} \times \reals^{n_2}.
    \end{equation*}
    Then
    \begin{equation*}
        \uncertaintySet \subseteq \proj_u(\eqSet_\uncertaintySet \cap \inEqSet) \iff \eqSet_\uncertaintySet \subseteq \inEqSet.
    \end{equation*}
\end{lemma}
\begin{proof}
    ``$\Rightarrow$'': Suppose $\uncertaintySet \subseteq \proj_u(\eqSet_\uncertaintySet \cap \inEqSet)$.
    Pick any $(u, x) \in \eqSet_\uncertaintySet$.
    Due to the projection, there exists $x'$ with $(u, x') \in  \eqSet_\uncertaintySet \cap \inEqSet$.
    The variable $x$ is uniquely determined for any $u \in \uncertaintySet$.
    Therefore $x = x' = \genEqSolFunc(u)$ holds and thus $(u, x) \in \inEqSet$.

    ``$\Leftarrow$'': Suppose $\eqSet_\uncertaintySet \subseteq \inEqSet$.
    Pick any $u \in \uncertaintySet$ and let $x = \genEqSolFunc(u)$.
    Then $(u, x) \in \eqSet_\uncertaintySet \subseteq \inEqSet$ and thus $(u, x) \in \eqSet_\uncertaintySet \cap \inEqSet$.
    This implies $u \in \proj_u(\eqSet_\uncertaintySet \cap \inEqSet)$.
\end{proof}
This lemma can be applied to all problems where a subset of the constraints defines a unique solution for each possible realization of the data.
Even if $\genEqSolFunc$ is only given implicitly by the solution of some (in-)equality system, the lemma is still applicable.

If \cref{asmp:G-is-func-graph} holds, \cref{lemma:feas-method-remove-proj} allows us to answer the original set containment problem~\eqref{u-proj-b} by deciding the equivalent set containment problem
\begin{equation*}
    \eqSet_\uncertaintySet \subseteq \inEqSet,
\end{equation*}
where \(\eqSet_\uncertaintySet = \condset{(u, x) \in \mathcal{U} \times \reals^{n_2}}{\genEqFunc(u, x) = 0} \) and $\inEqSet = \{(u, x) \in \reals^{n_1} \times \reals^{n_2} \mid \genIneqFunc_1(x) \geq 0,\ldots, \genIneqFunc_{k_2}\}(x) \geq 0\}$.
This set containment problem can then be decided with the optimization problems
\begin{equation}\label{prob:minconstraints}\tag{MinCons}
    \inf_{x \in \eqSet_\uncertaintySet} \genIneqFunc_i(x) \quad i =1,\ldots,k_2.
\end{equation}
The objective values of all $k_2$ optimization problems are non-negative if and only if $\eqSet_\uncertaintySet \subseteq \inEqSet$.
In cases where global optimality cannot be obtained easily, the criterion can be weakened by replacing the optimization problems \cref{prob:minconstraints} with relaxations since non-negative objective values of the relaxations imply non-negative objective values of the original problems.
However, this is only a sufficient criterion since $\eqSet_\uncertaintySet \subseteq \inEqSet$ might hold but at the same time some optimization problems can have negative objective values due to the relaxation.
% In order to develop a correct methodology for deciding robust feasibility in this manner, we need to make the following assumption:

The next lemma shows how the set containment question can still be decided if the considered sets are partitioned into subsets.
This will be important later when eliminating the absolute values of the gas transport problem.
\begin{lemma}\label{lemma:split-proj-into-subsets}
    Let $\uncertaintySet \subseteq \reals^{n_1}$ and let $\eqSet_\uncertaintySet,\,\inEqSet \subseteq \reals^{n_1}\times\reals^{n_2}$ with $\eqSet_\uncertaintySet = \condset{(u, x) \in \reals^{n_1}\times\reals^{n_2}}{x = \genEqSolFunc(u), \, u \in \uncertaintySet}$ for an arbitrary function $\genEqSolFunc\colon\reals^{n_1}\rightarrow\reals^{n_2}$.
    Let $\mathcal{S}_i$ ($i \in I$) be a collection of sets with $\mathcal{S}_i \subseteq \reals^{n_1}\times\reals^{n_2}$ such that $\bigcup_{i \in I} \mathcal{S}_i = \reals^{n_1}\times\reals^{n_2}$.
    Then
    \begin{align*}
        \uncertaintySet = \proj_u(\eqSet_\uncertaintySet) &\subseteq \proj_u(\eqSet_\uncertaintySet \cap \inEqSet)\\
        &\iff \nonumber\\
        \proj_u(\eqSet_\uncertaintySet \cap \mathcal{S}_i) &\subseteq \proj_u(\eqSet_\uncertaintySet \cap \mathcal{S}_i \cap \inEqSet)\quad\forall i \in I
    \end{align*}
\end{lemma}
\begin{proof}\begin{align*}%
        &\uncertaintySet = \proj_u(\eqSet_\uncertaintySet) \subseteq \proj_u(\eqSet_\uncertaintySet \cap \inEqSet) \\
        \overset{{\text{\cref{lemma:feas-method-remove-proj}}}}{\iff}&
        \eqSet_\uncertaintySet \subseteq \inEqSet \iff \eqSet_\uncertaintySet \cap \mathcal{S}_i \subseteq \inEqSet \cap \mathcal{S}_i\quad (\forall i \in I)
    \end{align*}%
    Let $\uncertaintySet'_i := \proj_u(\eqSet_\uncertaintySet \cap \mathcal{S}_i)$.
    Rewriting $\eqSet_\uncertaintySet \cap \mathcal{S}_i$ yields
    \begin{align*}
        \eqSet_\uncertaintySet \cap \mathcal{S}_i &= \condSet{(u, x)}{x = \genEqSolFunc(u),\, u\in\uncertaintySet,\,(u, x) \in \mathcal{S}_i} \\
        &= \condSet{(u, x)}{x = \genEqSolFunc(u),\, u\in\uncertaintySet,\,(u, \genEqSolFunc(u)) \in \mathcal{S}_i} \\
        &= \condSet{(u, x)}{x = \genEqSolFunc(u),\, u\in\condSet{u}{u \in \uncertaintySet,\, (u, \genEqSolFunc(u)) \in \mathcal{S}_i}}\\
        % &= \condSet{(u, x)}{x = \genEqSolFunc(u),\, u\in\condSet{u}{(u, x) \in A_\uncertaintySet \cap S_i}} \\
        &= \condSet{(u, x)}{x = \genEqSolFunc(u),\, u\in\proj_u(\eqSet_\uncertaintySet \cap \mathcal{S}_i)} \\
        &= \eqSet_{\proj_u(\eqSet_\uncertaintySet \cap S_i)} = \eqSet_{\uncertaintySet'_i}.
    \end{align*}
    Then
    \begin{align*}
         &\eqSet_\uncertaintySet \cap \mathcal{S}_i = \eqSet_{\uncertaintySet'_i} \subseteq \inEqSet \cap \mathcal{S}_i\quad (\forall i \in I) \\
         \overset{{\text{\cref{lemma:feas-method-remove-proj}}}}{\iff}&
         \uncertaintySet'_i = \proj_u(\eqSet_{\uncertaintySet'_i}) \subseteq \proj_u\Of{\eqSet_{\uncertaintySet'_i} \cap \inEqSet} = \proj_u\Of{\eqSet_{\uncertaintySet} \cap \inEqSet \cap \mathcal{S}_i} \quad (\forall i \in I).
    \end{align*}
\end{proof}

For a practical application, the optimization problems \cref{prob:minconstraints} need to be solved to global optimality.
As mentioned earlier, if global optimality cannot be ensured, a relaxation of the given problem can also suffice.
The structure of the optimization problems depends on the defining functions of $\eqSet_\uncertaintySet$, $\inEqSet$.
For the gas network problem, these typically are polynomials or piecewise polynomials.
Using the ideas of \cref{subsec:eliminate-abs}, the piecewise polynomial functions can be reformulation in terms of pure polynomials.
Instead of solving the resulting polynomial optimization problems \cref{prob:minconstraints}, sum of squares or moment relaxation of these problems are used instead.
These relaxations form a hierarchy of semidefinite programs, see \cite{Parrilo2003} and \cite{Lasserre2001}, respectively.

\section{Deciding Robustness for the Passive Gas Network Problem}\label{sec:deciding-robustness-of-gas}
In this section, the passive gas network problem under uncertainty is introduced.
It also contains crucial properties of the problem class as well as techniques for reduction of variables and procedures to eliminate the occurring absolute value functions.
Combined, these ideas allow a compact problem formulation as a polynomial feasibility system which will can be tackled using methods from \cref{sec:robust-problems}.

\subsection{The Passive Gas Network Problem}\label{subsec:passive-gas-net}
We consider a stationary passive gas network with horizontal pipes.
Gas can be inserted or withdrawn at each node of the network.
The goal is to decide whether a given set of demands can be satisfied by the network.
Even in the absence of uncertainties, this problem is challenging to solve since the resulting feasibility problem is in general nonlinear, \nonsmooth and \nonconvex.

\subsubsection{\Modeling the Nominal Passive Gas Network Feasibility Problem}
The network's topology is given by a weakly connected digraph $\graph=(\nodes^+, \arcs)$ with $\abs{\nodes^+} = \abs{\{0,\ldots,n+1\}} = n+1$ nodes and $\abs{\arcs} =m \geq n$ arcs.
The physical state of the network is represented by the (non-negative)  pressure $\pressureAtNode[\node] \in \realsPos$ at each node $\node \in \nodes$ and the flow $\massflowAtArc[\arc] \in \reals$ along each arc $\arc \in \arcs$.
Concerning the flow, a positive sign of $\massflowAtArc[\arc]$ indicates flow in edge direction, a negative sign the reverse.
Since the pressure only occurs in squared form, we introduce variables $\pressureAtNode[\node]^2 = \pressSqrAtNode[\node] \in \realsPos$  for the squared pressures, see \cref{eq:pressure-loss}.
Due to physical, technical and legal reasons, the squared pressures are bounded: $\pressSqrAtNode[\node] \in [\lbPressSqrAtNode[\node], \ubPressSqrAtNode[\node]]$, $\node \in \nodes^+$.
For a more comprehensive treatment of the gas transport problem, see e.g. \cite{Koch2015}.
A general survey on the problems arising in gas network operations is given in \cite{Rios-Mercado2015}.

Gas networks share a basic property with linear flow networks: at each node, flow conservation must hold.
Similar to the linear case, gas may be inserted or withdrawn at each node of the network.
This so called demand or \emph{nomination} is encoded in the vector $(\supplyAtNode[\node])_{\node \in \nodes^+}$ which has to be balanced: $\sum_{\node \in \nodes^+} \supplyAtNode[\node] = 0$.
Insertion is indicated by a positive sign, withdrawal by a negative sign of $\supplyAtNode[\node]$.
Flow conservation can then be stated as
\begin{equation}\label{eq:flow-cons}
    \sum_{\arc=\Arc{\node}{\otherNode} \in \arcs} \massflowAtArc[\arc] - \sum_{\arc=\Arc{\otherNode}{\node} \in \arcs} \massflowAtArc[\arc] = \supplyAtNode[\node], \quad \forall \node \in \nodes^+.
\end{equation}
So far, the model is identical to a regular linear network flow problem.
More complexity in the form of nonlinear constraints is introduced once the physical laws of gas transport are considered.

According to the Weymouth Equation \cref{eq:pressure-loss} (see \cite{Weymouth1912}), when gas flows through a pipe, its pressure decreases.
The difference of the squared pressures at both ends of the pipe is proportional to the signed squared flow along the pipe.
The magnitude of the pressure drop is influenced by the pipe's pressure loss factor  $\plossCoeffAtArc[\arc]$, which (amongst other factors) depends upon the length, diameter and roughness of the pipe.
A more in depth look at the different \modeling approaches for the pressure loss factor can be found in \cite{Koch2015} and \cite{Pfetsch2015}.

By defining $f(x) := x\abs{x}$, the pressure loss relation can be expressed as
\begin{equation}\label{eq:pressure-loss}
    \pressSqrAtNode[\node] - \pressSqrAtNode[\otherNode] = \plossCoeffAtArc[\arc] \abs{\massflowAtArc[\arc]}\massflowAtArc[\arc] = \plossCoeffAtArc[\arc] f(\massflowAtArc[\arc]), \quad \forall \arc=\Arc{\node}{\otherNode} \in \arcs.
\end{equation}

Let $\nodeArcIncidentFull \in \reals^{\abs{\nodes^+} \times \abs{\arcs}}$ be the node-arc-incidence matrix of $\graph$, that is $(\nodeArcIncidentFull)_{\arc \node} = +1$ and $(\nodeArcIncidentFull)_{\arc \otherNode} =-1$ for \mbox{$\arc=\Arc{\node}{\otherNode} \in \arcs$}.
With $\nodeArcIncidentFull$, the flow conservation \cref{eq:flow-cons} can be stated in a more compact manner:
\begin{equation}\label{eq:flow-cons-compact}
    \nodeArcIncidentFull \massflow = \fullSupplyVec.
\end{equation}
By defining $\PlossCoeff[\plossCoeff] := \diag\Of{\plossCoeffAtArc[1], \dots, \plossCoeffAtArc[{\abs{\arcs}}]}$ and $F(\massflow):= \transpose{((f(\massflowAtArc[1]), \dots, f(\massflowAtArc[{\abs{\arcs}}]))}$, the pressure loss constraints \cref{eq:pressure-loss} can be combined to
\begin{equation}\label{eq:pressure-loss-compact}
    \transpose{\nodeArcIncidentFull} \pressSqr = - \PlossCoeff[\plossCoeff] F(\massflow).
\end{equation}
With \cref{eq:flow-cons-compact,eq:pressure-loss-compact}, the feasibility problem can be stated as a potential driven network problem
\begin{equation}\tag{PotN}\label{model:potn}
    \begin{aligned}
        \nodeArcIncidentFull \massflow  &= \fullSupplyVec,\\
        \transpose{\nodeArcIncidentFull} \pressSqr &= - \PlossCoeff[\plossCoeff] F(\massflow),  \\
        \pressSqr &\in [\lb{\pressSqr}, \ub{\pressSqr}],\\
        \massflow &\in \reals^{\abs{\arcs}}.
    \end{aligned}
\end{equation}

\subsubsection{Reduction of Variables}\label{subsubsecn:reduction-of-variables}
By a result of \cite{Gotzes2016}, all pressure variables and $\abs{\nodes^+}-1$ of the flow variables can be eliminated from the system.
It is well known that for connected graphs, $\nodeArcIncidentFull$ has rank $\abs{\nodes^+}-1$ and an arbitrary row can be removed while preserving the set of solutions of \cref{eq:flow-cons}.
For ease of notation, we discard the row corresponding to node $0$ and obtain $\nodeArcIncident$ from $\nodeArcIncidentFull$ in this way.
The set $\nodes = \nodes^+ \minus \{0\} = \{1,\ldots,n\}$ of nodes and the demand vector $\fullSupplyVec$ are adjusted accordingly.

\begin{theorem}[\cite{Gotzes2016}]
    Let $\nodeArcIncident$ be the node-arc-incidence-matrix of a graph $\graph$ as described above and let $\nodeArcIncident = (\nodeArcIncidentBasis, \nodeArcIncidentNonBasis)$ be partition into basis and non basis submatrices of $\nodeArcIncident$.
    Let $(\PlossCoeffBasis[\plossCoeff]$, $\PlossCoeffNonBasis[\plossCoeff])$, $(F_B$, $F_N)$, and $(\massflowBasis$, $\massflowNonBasis)$ be the corresponding partitions of $\PlossCoeff[\plossCoeff]$, $F$, and $\massflow$, respectively.
    Define
    \begin{equation*}\label{eq:def-agg-ploss-g}
        g\colon\reals^{|\arcs|}\times\reals^{\abs{N}} \rightarrow \reals^{|\nodes|}, \quad g(\plossCoeffBasis, \massflowNonBasis) := \OF{\nodeArcIncidentBasisTransposed}^{-1} \PlossCoeffBasis[\plossCoeff] F_B\OF{\nodeArcIncidentBasisInverse \OF{\supplyVec -\nodeArcIncidentNonBasis \massflowNonBasis} }.
    \end{equation*}

    Then the model \cref{model:potn} is equivalent to the following reduced model in variables $\massflowNonBasis$:
    \begin{equation}\tag{RPotN}\label{model:rpotn}
      \begin{aligned}
          \nodeArcIncidentNonBasisTransposed g(\plossCoeff, \massflowNonBasis) &= \PlossCoeffNonBasis[\plossCoeff] F_N(\massflowNonBasis) \\
          \lbPressSqrAtNode[0] &\leq \min_{i=1,\ldots,n}\left[\ubPressSqrAtNode[i] + g_i(\plossCoeff, \massflowNonBasis)\right] \\
          \ubPressSqrAtNode[0] &\geq \max_{i=1,\ldots,n}\left[\lbPressSqrAtNode[i] + g_i(\plossCoeff, \massflowNonBasis)\right] \\
          \min_{i=1,\ldots,n}\left[\ubPressSqrAtNode[i] + g_i(\plossCoeff, \massflowNonBasis)\right] &\geq \max_{i=1,\ldots,n}\left[\lbPressSqrAtNode[i] + g_i(\plossCoeff, \massflowNonBasis)\right] \\
          \massflowNonBasis &\in \reals^{\abs{N}},
      \end{aligned}
    \end{equation}
    where $\lbPressSqrAtNode[0]$, $\ubPressSqrAtNode[0]$ are the squared pressure bounds at the root node, respectively.\\
    If a feasible $\massflowNonBasis$ for \cref{model:rpotn} exists, the remaining variables $\massflowBasis, \pressSqr$ can be recovered through $\massflowBasis = \nodeArcIncidentBasisInverse \OF{ \supplyVec - \nodeArcIncidentNonBasis  \massflowNonBasis }$ and $\pressSqrAtNode[i] = \pressSqrAtNode[0] - g_i(\plossCoeff, \massflowNonBasis)$ ($i=1,\dots,n$).
    The value of $\pressSqrAtNode[0]$ is an arbitrary given element of
    \begin{equation*}
        \left[
            \max_{i=1,\dots,n}[\lbPressSqrAtNode[i] + g_i(\plossCoeff, \massflowNonBasis)],
            \min_{i=1,\dots,n}[\ubPressSqrAtNode[i] + g_i(\plossCoeff, \massflowNonBasis)]
        \right].
    \end{equation*}
    Conversely, a vector $\massflowNonBasis$ that was extracted from a solution $\optimal{\massflow}$, $\optimal{\pressSqr}$ of \cref{model:potn} is feasible for \cref{model:rpotn}.
\end{theorem}
Depending on the situation, it can be beneficial to consider the reduced problem \cref{model:rpotn} or the original problem \cref{model:potn}.
For that purpose, let
\begin{align*}
    \flowExtensionMap &\colon \reals^{\abs{N}} \to \reals^{\abs{A}}, \\
    \flowExtensionMap_\arc(\massflowNonBasis) &:=
        \begin{cases}
            \OF{\nodeArcIncidentBasisInverse \OF{ \supplyVec - \nodeArcIncidentNonBasis  \massflowNonBasis }}_\arc, & \text{ if } \arc \in B,\\
            \OF{\massflowNonBasis}_\arc , & \text{ if } \arc \in N.
        \end{cases}
\end{align*}
This affine linear function maps cycle flow values to flows on all arcs of the graph.
For graphs with a single cycle, $\flowExtensionMap$ can be simplified to $\flowExtensionMap_\arc(\massflowNonBasis) = \massflowNonBasis - \beta_\arc$ for some $\beta_\arc \in \reals$.

\subsubsection{Uniqueness of Flow}
Another important result in this context concerns the structure of the feasible set of \cref{model:potn}.
As shown in \cite{Collins1978, Rios-Mercado2002}, the feasible flow of a given demand scenario for a network without pressure bounds is uniquely determined.
\begin{theorem}[\cite{Collins1978}]\label{thm:unique-flow}
    Consider \eqref{model:potn} without pressure bounds.
    Then for fixed $\plossCoeff \in \realsStrictPos^{\abs{\arcs}}$, the solution space has the following properties:
    \begin{enumerate}
        \item The projection on the flow variable $\massflow$ contains a single point, i.e. the flow is unique.
        \item The projection on the squared pressure variable $\pressSqr$ has the form
        \begin{equation*}
            \condSet{\optimal{\pressSqr} + \eta\transpose{(1, \dots, 1)}}{\eta \in \reals}.
        \end{equation*}
    \end{enumerate}
\end{theorem}
In case of pressure bounds, the variable $\eta$ is constrained:
\begin{gather*}
    \eta \in [\lb{\eta}, \ub{\eta}]
    \quad\text{ with } \quad
    \lb{\eta} := \max_{\node\in\nodes}\Of{\lbPressSqrAtNode[\node] - \optimal{\pressSqrAtNode[\node]}} \quad \text{and} \quad
    \ub{\eta} := \min_{\node\in\nodes}\Of{\ubPressSqrAtNode[\node] - \optimal{\pressSqrAtNode[\node]}}.
\end{gather*}
As a simple consequence, if the pressure of a feasible problem is fixed at any node, the pressure values at the remaining network nodes are also uniquely determined.

\subsubsection{The Passive Gas Network Problem Under Uncertainty}
Based on this nominal setting \cref{model:potn,model:rpotn}, uncertainty is introduced into the problem.
Disregarding any combinatorial uncertainties (e.g. random failing of arcs), two possible sources of uncertainty are present in the given model: fluctuations in the demand $\supplyAtNode[\node]$ and variations of the pressure loss factor $\plossCoeffAtArc[\arc]$.
In this paper, we focus on uncertainties in the pressure loss coefficient.
The value of $\plossCoeffAtArc[\arc]$ is influenced by specific chemical properties of the gas as well as physical parameters of the pipe like e.g. its length, diameter and roughness.
In particular, the roughness value of the pipe's wall changes during the network's operation due to aging effects and accumulation of dirt.
It is difficult to measure this parameter after the network begins operation.
Since the roughness values can only be estimated, a robust treatment of the problem is reasonable.
The goal of robust optimization is to immunize solutions of an optimization problem against a set of parameters which can be realized from a given uncertainty set.
The problem is required to be solvable for all possible realization of the uncertainty.

It is assumed that the pressure loss factor of each pipe is strictly positive and lies within some a-priori known interval
\begin{equation*}
    \plossCoeffAtArc[\arc] \in [\lbPlossCoeffAtArc[\arc], \ubPlossCoeffAtArc[\arc]] \subseteq \realsStrictPos \quad \forall \arc \in \arcs
\end{equation*}
with $0 < \lbPlossCoeffAtArc[\arc] \leq \ubPlossCoeffAtArc[\arc]$.
Furthermore, possible correlation between different pipes is ignored.
The resulting uncertainty set $\uncertaintySet$ is therefore given by the hyperrectangle
\begin{align*}
    \uncertaintySet := \times_{\arc \in \arcs} [\lbPlossCoeffAtArc[\arc], \ubPlossCoeffAtArc[\arc]].
\end{align*}

% Recall the unique flow theorem of the last paragraph.
By \cref{thm:unique-flow}, a problem without pressure bounds always admits a uniquely determined flow that satisfies the given demand.
Parameterizing this result by the pressure loss factors motivates the following corollary:
\begin{corollary}\label{cor:unique-flow}
    For networks without pressure bounds, there exists a function
    \begin{align*}
        \tilde{\massflow} \colon \realsStrictPos^{\abs{\arcs}} &\rightarrow \reals^{\abs{N}}\\
        \plossCoeff &\mapsto \tilde{\massflow}(\plossCoeff)
    \end{align*}
    that solves
    \begin{equation*}
        \nodeArcIncidentNonBasisTransposed g(\plossCoeff, \tilde{\massflow}(\plossCoeff)) = \PlossCoeffNonBasis[\plossCoeff] F_N(\tilde{\massflow}(\plossCoeff))
    \end{equation*}
    for all $\plossCoeff \in \realsStrictPos^{\abs{\arcs}}$.
\end{corollary}

\subsubsection{Deciding Robustness of the Gas Network Problem}\label{sec:gas-definitions}
Using model \cref{model:rpotn}, let
\begin{align*}
    \eqSet &:= \condSet{(\plossCoeff, \massflowNonBasis) \in \realsStrictPos^{\abs{\arcs}} \times \reals^{\abs{N}}}{\nodeArcIncidentNonBasisTransposed g(\plossCoeff, \massflowNonBasis) - \PlossCoeffNonBasis[\plossCoeff] F(\massflowNonBasis) = 0}, \\
    \eqSet_\uncertaintySet &:= \condSet{(\plossCoeff, \massflowNonBasis) \in \eqSet}{\plossCoeff \in \uncertaintySet}
\end{align*}
and
\begin{align*}
    \inEqSet :=  \condSet{(\plossCoeff, \massflowNonBasis) }{%
        \begin{aligned}
             - \lbPressSqrAtNode[0] + \ubPressSqrAtNode[i] + g_i(\plossCoeffBasis, \massflowNonBasis) \geq 0,&&  i &\in \{1,\ldots,n\}\\
               \ubPressSqrAtNode[0] - \lbPressSqrAtNode[i] - g_i(\plossCoeffBasis, \massflowNonBasis) \geq 0,&&  i &\in \{1,\ldots,n\}\\
               \ubPressSqrAtNode[i] + g_i(\plossCoeffBasis, \massflowNonBasis) - \lbPressSqrAtNode[j] - g_j(\plossCoeffBasis, \massflowNonBasis) \geq 0, && i,j &\in \{1,\ldots,n\} \\
        (\plossCoeff, \massflowNonBasis) \in \realsStrictPos^{\abs{\arcs}} \times \reals^{\abs{N}}&&&
        \end{aligned}
    }.
\end{align*}
The set $\eqSet$ (resp. $\eqSet_\uncertaintySet$) contains all feasible combinations $\plossCoeff, \massflowNonBasis$ (resp. with $\plossCoeff \in \uncertaintySet$) arising from the cycle flow equations.
Due to \cref{cor:unique-flow}, this set can be stated equivalently as the graph of $\massflow(\plossCoeff)$.
% The set $\eqSet_\uncertaintySet$ describes all feasible combinations when restricted to a fixed uncertainty set $\uncertaintySet$.
On the other hand, $\inEqSet$ can be seen as all combinations $\plossCoeff, \massflowNonBasis$ that are feasible for the given pressure bounds.

Combining both $\eqSet$ and $\inEqSet$, let
\begin{equation*}
    \feasCoeffsAndFlows := \eqSet \cap \inEqSet
\end{equation*}
be the set of all feasible uncertainty/solution pairs of the given gas transport problem.

% Let \feasCoeffsAndFlows be the set of all feasible $\plossCoeff, \massflowNonBasis$ of \eqref{model:rpotn}:
% \begin{equation}
%   \feasCoeffsAndFlows :=
%   \condSet{
%   \plossCoeff \in \realsPos^{\abs{\arcs}}, \massflowNonBasis \in \reals^{\abs{N}}
%   }{
%   \begin{aligned}
%       \nodeArcIncidentNonBasisTransposed g(\massflowNonBasis, \plossCoeff) &= \PlossCoeffNonBasis[\plossCoeff] F(\massflowNonBasis) \\
%       \lbPressSqrAtNode[0] &\leq \min_{i=1,\ldots,n}\left[\ubPressSqrAtNode[i] + g_i(\massflowNonBasis, \plossCoeff)\right] \\
%       \ubPressSqrAtNode[0] &\geq \max_{i=1,\ldots,n}\left[\lbPressSqrAtNode[i] + g_i(\massflowNonBasis, \plossCoeff)\right] \\
%       \min_{i=1,\ldots,n}\left[\ubPressSqrAtNode[i] + g_i(\massflowNonBasis, \plossCoeff)\right] &\geq \max_{i=1,\ldots,n}\left[\lbPressSqrAtNode[i] + g_i(\massflowNonBasis, \plossCoeff)\right] \\
%       \massflowNonBasis &\in \reals^{\abs{N}}
%   \end{aligned}
%   }
% \end{equation}
The task is now to decide whether the network allows a feasible flow for all $\plossCoeff \in \uncertaintySet$.\\
Let $\proj_\plossCoeff(\feasCoeffsAndFlows)$ be the projection of the feasible pairs of pressure loss coefficients and flows onto the space of the uncertainty set.
This set contains all pressure loss coefficients which admit a feasible flow in the corresponding problem.
In this context, deciding robustness with respect to $\uncertaintySet$ is equivalent to checking whether the uncertainty set $\uncertaintySet$ is contained in the projection $\proj_\plossCoeff(\feasCoeffsAndFlows)$:
\begin{equation*}
    \uncertaintySet \subseteq \proj_\plossCoeff(\feasCoeffsAndFlows) = \condSet{\plossCoeff}{\exists\, \massflowNonBasis\colon (\plossCoeff, \massflowNonBasis) \in \feasCoeffsAndFlows}.
\end{equation*}

\subsection{Deciding Robust Feasibility on Tree Networks}\label{subsec:robust-problems-tree}
Consider a network whose underlying topology is a tree, i.e. a connected, cycle-free graph.
Since there are no cycles and therefore $N = \emptyset$, the description of the feasible set $\feasCoeffsAndFlows$ does not contain any cycle flow variables $\massflowNonBasis$.
Since there are no flow variables, the function $g(\plossCoeff, \massflowNonBasis)$ as defined in \cref{eq:def-agg-ploss-g} is reduced to a function of the form
\begin{equation*}
    g(\plossCoeff) = \OF{\nodeArcIncidentBasisTransposed}^{-1} \PlossCoeffBasis[\plossCoeff] F_B(\nodeArcIncidentBasisInverse \supplyVec).
\end{equation*}
From this description, we can see that $g(\plossCoeff)$ is a linear function of $\plossCoeff$.
Note that $F_B(\nodeArcIncidentBasisInverse \supplyVec)$ is a constant expression that can be calculated in advance.

With $N = \emptyset$, the set% With no $\massflowNonBasis$ variables left, the feasible set \feasCoeffsAndFlows is thus polyhedral in the admissible pressure loss factors $\plossCoeff$:
\begin{equation}\label{eq:tree-feas-phi}
    \feasCoeffsAndFlows =
    \condSet{
    \plossCoeff \in \realsPos{\abs{\arcs}}
    }{
    \begin{aligned}
        \lbPressSqrAtNode[0] &\leq \min_{i=1,\ldots,n}\left[\ubPressSqrAtNode[i] + g_i( \plossCoeff)\right] \\
        \ubPressSqrAtNode[0] &\geq \max_{i=1,\ldots,n}\left[\lbPressSqrAtNode[i] + g_i( \plossCoeff)\right] \\
        \min_{i=1,\ldots,n}\left[\ubPressSqrAtNode[i] + g_i( \plossCoeff)\right] &\geq \max_{i=1,\ldots,n}\left[\lbPressSqrAtNode[i] + g_i( \plossCoeff)\right] \\
    \end{aligned}
    }
\end{equation}
is polyhedral since all $g_i(\plossCoeff)$ are linear and the $\min$ / $\max$ expressions can be replaced by a finite number of linear constraints.

In this case, checking robust feasibility with respect to a given polyhedral uncertainty set \uncertaintySet is equivalent to deciding the set containment problem
\begin{equation*}
    \uncertaintySet \subseteq \proj_\plossCoeff(\feasCoeffsAndFlows) = \feasCoeffsAndFlows
\end{equation*}
for two polyhedra $\uncertaintySet$ and $\feasCoeffsAndFlows$.
As the following lemma by \cite{Mangasarian2002} shows, this can be done efficiently with LP duality:
\newcommand{\fstMatrix}{\TRRNotation{S}\xspace}
\newcommand{\sndMatrix}{\TRRNotation{T}\xspace}
\newcommand{\thrdMatrix}{\TRRNotation{W}\xspace}
\newcommand{\fstRhs}{\vektor{s}\xspace}
\newcommand{\sndRhs}{\vektor{t}\xspace}
\begin{lemma}[\cite{Mangasarian2002}]\label{lemma:polyhedral-set-containmet}
    Let the set $\subSet := \{x \mid \fstMatrix x \geq \fstRhs\}$ and let $\superSet := {\{x \mid \sndMatrix x  \leq \sndRhs\}}$, where $\fstMatrix \in \reals^{m \times n}$, $\sndMatrix \in \reals^{k \times n}$ and let $\sndMatrix$ be nonempty.
    Then the following are equivalent:
    \begin{enumerate}
        \item $\superSet \subseteq \subSet$, that is:
            \begin{equation*}
                \sndMatrix x \leq \sndRhs \implies \fstMatrix x \geq \fstRhs.
            \end{equation*}
        \item For $i = 1,\dots,m$, the $m$ linear programs are solvable and satisfy:
            \begin{equation*}
                \min_x \{ (\fstMatrix x)_i \mid \sndMatrix x \leq \sndRhs \} \geq \fstRhs_i.
            \end{equation*}
        \item There exists a matrix $\thrdMatrix \in \reals^{m \times k}$ such that:
            \begin{equation*}\label{eq:polyhedral-set-containment-matrix}
                \fstMatrix  + \thrdMatrix \sndMatrix = 0, \, \fstRhs + \thrdMatrix \sndRhs \leq 0, \,\thrdMatrix \geq 0.
            \end{equation*}
    \end{enumerate}
\end{lemma}
\begin{proof}
    See \cite{Mangasarian2002}.
\end{proof}
\begin{corollary}\label{cor:robust-tree}
    Let $\uncertaintySet = \condSet{\plossCoeff}{\sndMatrix \plossCoeff \leq \sndRhs}$ be a polyhedral uncertainty set.
    Let $\feasCoeffsAndFlows = \condSet{\plossCoeff}{\fstMatrix \plossCoeff \geq \fstRhs}$ be the polyhedral set of feasible pressure loss factors $\plossCoeff$ for a gas transport problem over a tree-shaped network.

    Then robustness with respect to $\uncertaintySet$ can be decided by solving a linear program.
\end{corollary}

\subsubsection{Robust Feasibility of Tree Networks as a Function of a Node's Pressure}
\Cref{cor:robust-tree} allows us to characterize robustness of a tree network in terms of the pressure at an arbitrary chosen node.
Let $\graph = (\nodes, \arcs)$ be the graph of a tree network.
Without loss of generality, we select the tree's root node $0$ as basis of our considerations.
Suppose the pressure value at this node is fixed, i.e. $\pressSqrAtNode[0] := \lbPressSqrAtNode[0] = \ubPressSqrAtNode[0]$.
Our aim is to specify all $\pressSqrAtNode[0]$ such that the gas network problem is robust feasible.

As can be inferred from \cref{eq:tree-feas-phi}, the pressure bounds only appear as constants in the linear inequality constraints.
With the conventions of the previous corollary, the set of feasible pressure loss coefficients can thus be expressed in terms of the root node's pressure $\pressSqrAtNode[0]$:
\begin{equation*}
    \feasCoeffsAndFlows(\pressSqrAtNode[0]) =
    \condSet{\plossCoeff}{\fstMatrix \plossCoeff \geq \fstRhs(\pressSqrAtNode[0])}.
\end{equation*}
The right hand side $s$ of the linear inequality system is a linear function $s\colon\reals \rightarrow \reals^{\abs{\arcs}}$ of $\pressSqrAtNode[0]$.
Applying \cref{lemma:polyhedral-set-containmet} to the set containment question $\uncertaintySet \subseteq \feasCoeffsAndFlows(\pressSqrAtNode[0])$ yields
\begin{align*}
    &\uncertaintySet = \condSet{\plossCoeff}{\sndMatrix \plossCoeff \leq \sndRhs} \subseteq \condSet{\plossCoeff}{\fstMatrix \plossCoeff \geq \fstRhs(\pressSqrAtNode[0])} = \feasCoeffsAndFlows(\pressSqrAtNode[0]) \\
    \iff& \mathcal{X}(\pressSqrAtNode[0]) := \condSet{\thrdMatrix \in \realsPos^{m \times k}}{
        \begin{aligned}
            \fstMatrix  + \thrdMatrix \sndMatrix &= 0 \\
            \fstRhs(\pressSqrAtNode[0]) + \thrdMatrix \sndRhs &\leq 0
        \end{aligned}} \neq \emptyset.
\end{align*}

\begin{lemma}\label{lemma:robust-tree-psqr-interval}
    Given a tree network $\graph=(\nodes, \arcs)$ with an arbitrary root node 0 and a polyhedral uncertainty set $\uncertaintySet$.
    Then the network is robust feasible if and only if the root node's squared pressure satisfies
    \begin{equation*}
        \pressSqrAtNode[0] \in [\optimal{\lbPressSqrAtNode[0]}, \optimal{\ubPressSqrAtNode[0]}]
    \end{equation*}
    with $\optimal{\lbPressSqrAtNode[0]}$, $\optimal{\ubPressSqrAtNode[0]}$ being optimal values of the linear programs
    \begin{subequations}\label{eq:robust-tree-psqr-interval-lps}
        \begin{align}
            \optimal{\lbPressSqrAtNode[0]} &:= \min_{\pressSqrAtNode[0], \thrdMatrix} \pressSqrAtNode[0] \text{ s.t. } \thrdMatrix \in \mathcal{X}(\pressSqrAtNode[0]),  \label{eq:robust-tree-psqr-interval-lps:a}\\
            \optimal{\ubPressSqrAtNode[0]} &:= \max_{\pressSqrAtNode[0], \thrdMatrix} \pressSqrAtNode[0] \text{ s.t. } \thrdMatrix \in \mathcal{X}(\pressSqrAtNode[0]).  \label{eq:robust-tree-psqr-interval-lps:b}
        \end{align}
    \end{subequations}
    \begin{proof}
    The set $\condSet{(\pressSqrAtNode[0], \thrdMatrix)}{\thrdMatrix \in \mathcal{X}(\pressSqrAtNode[0])}$ is polyhedral and thus convex.
    Therefore, the set of all feasible $\pressSqrAtNode[0]$ is the interval
    \begin{equation*}
        [\optimal{\lbPressSqrAtNode[0]}, \optimal{\ubPressSqrAtNode[0]}]
    \end{equation*}
    whose endpoints are the optimal values of the linear programs \cref{eq:robust-tree-psqr-interval-lps:a,eq:robust-tree-psqr-interval-lps:a}.
    % \begin{align}
    %     \optimal{\lbPressSqrAtNode[0]} &:= \min_{\pressSqrAtNode[0], \thrdMatrix} \pressSqrAtNode[0] \text{ s.t. } \thrdMatrix \in \mathcal{X}(\pressSqrAtNode[0]),  \\
    %     \optimal{\ubPressSqrAtNode[0]} &:= \max_{\pressSqrAtNode[0], \thrdMatrix} \pressSqrAtNode[0] \text{ s.t. } \thrdMatrix \in \mathcal{X}(\pressSqrAtNode[0]).
    % \end{align}
    \end{proof}
\end{lemma}

\subsection{Eliminating the Absolute Value Functions}\label{subsec:eliminate-abs}
In order to apply tools from polynomial optimization to the gas network problem, the constraining functions of \feasCoeffsAndFlows have to be converted to a polynomial representation.
Currently, the pressure drop equations
\begin{equation*}
    \pressSqrAtNode[\node] - \pressSqrAtNode[\otherNode] = \plossCoeff \massflowAtArc[\arc] \abs{\massflowAtArc[\arc]} = \plossCoeff f(\massflowAtArc[\arc])
\end{equation*}
introduce absolute values in the problem.
After elimination of the absolute values, \feasCoeffsAndFlows is transformed from a piecewise polynomial representation to an equivalent but purely polynomial description.
Depending on the topology of a given instance, it may be possible to eliminate a lot of absolute values in advance since all arcs which are not part of a cycle have fixed flow direction.
For example, in the case of tree networks, all directions are known in advance.
Apart from that, the flow direction can be fixed by other preprocessing algorithms, \eg flow/pressure propagation or bound tightening methods.
Further discussion on that topic can be found in \cite{Geissler2011}.

This chapter presents three different methods for the elimination of absolute values.
First, a technique from mixed-integer optimization is employed to model absolute values using binary variables.
With this method, both the feasibility and the infeasibility method can be used.
Next, the implications of straight forward case distinction are discussed.
In general, this technique can only be used for the feasibility method as will be later explained.
Finally, the case distinction idea is further investigated for networks which contain a single cycle.
In this setting, the absolute values can be eliminated by restricting the uncertainty set to polyhedral subsets.
It is shown how the overall problem can be decomposed into linearly many subproblems which can be decided with both methods.

\subsubsection{Elimination by Auxiliary Binary Variables}
By introducing additional binary variables, the absolute value functions can be eliminated.
This technique is very similar to what is typically done in mixed-integer optimization.
We demonstrate the idea using the example of $\abs{x}{x}$.
Assume that $\abs{x}$ is bounded: $\abs{x} \leq M$.
This is a natural assumption since the flows within the network cannot become arbitrary large.
With the introduction of a new binary variable $b$, the signed-square expression $y = \abs{x}{x}$ can be stated equivalently using polynomials via
\begin{align*}
        y &= (2b - 1) x^2, \\
        (-1+b)M \leq &x \leq bM, \\
        b &= b^2.
\end{align*}
Applying this construction to each absolute value function on each arc $\arc \in \arcs$ yields a purely polynomial description of $\feasCoeffsAndFlows$ that can be used in the feasibility and infeasibility methods.

% Note that compared to the previous paragraphs, the optimization problems contain $\abs{\arcs}$ additional variables, $4\abs{\arcs}$ additional constraints and uses polynomials of degree up to three instead of two.

\subsubsection{Elimination by Case Distinction: the General Case}
Using the original problem definition \cref{model:potn}, each pipe $\arc \in \arcs$ introduces an absolute value with its pressure loss equation.
In general, one might expect that by eliminating each absolute value function, the problem is split into $2^{\abs{\arcs}}$ cases.
This paragraph shows how the number of cases mainly depends on the amount of fundamental cycles in the graph and thus can be much smaller than $2^{\abs{\arcs}}$.
We remark that the following results identify the feasible flow directions in a linear network flow model instead of the gas transport problem.
However, this is no restriction since adding constraints concerning the gas physics reduces the number of possible cases even further.

Due to \cref{lemma:split-proj-into-subsets}, the overall set containment problem can be decided by splitting the problem into a series of subproblems.
Each subproblem arises by restricting the original problem to certain subsets, \eg to orthants of $\reals^{\abs{\arcs}}$ for the absolute value case distinction.
Let $\mathcal{O}_1, \ldots, \mathcal{O}_{2^{\abs{\arcs}}} = \{\realsPos, \realsNeg\}^{\abs{\arcs}}$ be the set of orthants in $\reals^{\abs{\arcs}}$.
In the original model \cref{model:potn}, the additional constraint $\massflow \in \mathcal{O}_i$ restricts the flow to a specific orthant and allows the elimination of all absolute value functions.
In the reduced model, the variables $\massflowBasis$ are replaced by $\massflowBasis = \nodeArcIncidentBasisInverse \left(\supplyVec -\nodeArcIncidentNonBasis \massflowNonBasis\right)$.
The transformed case distinction is
\begin{equation*}
    \begin{pmatrix}\massflowBasis \\ \massflowNonBasis\end{pmatrix} =
    \begin{pmatrix}\nodeArcIncidentBasisInverse \left(\supplyVec -\nodeArcIncidentNonBasis \massflowNonBasis\right) \\ \massflowNonBasis \end{pmatrix} \in \mathcal{O}_i.
\end{equation*}

By considering the reduced model, the next proposition shows that the number of case distinctions mainly depends on the amount of fundamental cycles in the graph.
\begin{proposition}
    Let $\graph$ be a connected digraph with $\abs{A}$ arcs and $\abs{N}$ fundamental cycles.
    Then there can be at most $\sum_{i=0}^{\abs{N}} \binom{\abs{A}}{i} \in \mathcal{O}(\abs{A}^{\abs{N}})$ many feasible flow directions in the network.
    The corresponding subproblems can be constructed in \runtime $\mathcal{O}(\abs{A}^{\abs{N}})$.
\end{proposition}
\begin{proof}
    The problem of finding all feasible flow directions can be reduced to a problem concerning the arrangement of hyperplanes.
    For ease of exposition, consider the \nonnegative orthant $\mathcal{O}^+ = \realsPos^{\abs{A}}$.
    Using the flow function $\flowExtensionMap(\cdot)$ as defined in \cref{subsubsecn:reduction-of-variables}, fixing the flow direction to this orthant amounts to the constraint $\flowExtensionMap(\massflowNonBasis) \in \mathcal{O}^+$, \ie $\flowExtensionMap(\massflowNonBasis) \geq 0$.
    Each entry of $\flowExtensionMap(\cdot)$ defines a hyperplane in $\reals^{\abs{N}}$.
    Consider the regions that can arise by segmenting $\reals^{\abs{N}}$ using the hyperplanes in $\flowExtensionMap(\cdot)$.
    For all $\arc \in \arcs$, each region is a subset of either $\flowExtensionMap_\arc(\massflowNonBasis) < 0$ or $\flowExtensionMap_\arc(\massflowNonBasis) > 0$.
    Therefore, the flow direction on all arcs in the graph is constant on each region.
    The total number of regions that can be constructed in $\reals^{\abs{N}}$ using $\abs{A}$ hyperplanes is bounded by $\sum_{i=0}^{\abs{N}} \binom{\abs{A}}{i} \in \mathcal{O}(\abs{A}^{\abs{N}})$ (\cite{Zaslavsky1975}).
    Furthermore, constructing all regions can be achieved in \runtime $\mathcal{O}(\abs{A}^{\abs{N}})$ using the algorithm of \cite{Edelsbrunner1986}.
\end{proof}

However, there is an issue arising with this approach as the subproblems are of the type (see \cref{lemma:split-proj-into-subsets})
\begin{equation*}
    \proj_\phi(\eqSet_\uncertaintySet \cap \mathcal{O}_i) \subseteq \proj_\phi(\eqSet_\uncertaintySet \cap \inEqSet \cap \mathcal{O}_i).
\end{equation*}
The feasibility method can be employed as-is since optimizing over a projection set poses no restriction.
On the other hand, the infeasibility method can not be applied as easily since it requires  the moments over the uncertainty set.
In case of the given subproblems, this is the set $\proj_\phi(\eqSet_\uncertaintySet \cap \mathcal{O}_i)$.
In general, it is unclear how the moments can be obtained without explicitly constructing the projection.
Nevertheless, this is possible  for networks with one cycle.
The next section gives the description of $\proj_\phi(\eqSet_\uncertaintySet \cap \mathcal{O}_i)$ for this case.
In this setting, the infeasibility method can be applied since the projected set is polyhedral.

\subsubsection{Elimination by Case Distinction: a Shortcut for Networks with One Cycle}\label{sec:elim-cycle-abs}
On networks with only one cycle, a considerable simplification can be applied.
The absolute values can be eliminated by restricting the problem to certain subsets of the uncertainty set.
In contrast, the previous case distinction method relied on restricting the flow variables.
The advantage of using subsets of the uncertainty set for this purpose is that the infeasibility method can be applied as well since it requires explicit knowledge of the uncertainty set.

For the purpose of this chapter, we assume a directed cyclic graph where each arc points to a different node:
\begin{assumption}\label{asmp:cyclic-connected-graph}
    Let $\graph=(\nodes, \arcs)$ be a directed cyclic graph with $\nodes=\{0,\dots,n\}$, $\arcs=\{(0,1),(1,2),\dots,(n-1,n),(n,0)\}$, and nonzero demand $\fullSupplyVec \in \reals^{\abs{\nodes}}$.
\end{assumption}
Due to the cyclic structure, the arcs can be uniquely identified by their first node.
We assume the last edge to be part of the nonbasis, thus there is only one problem variable $\massflowAtArc[n] \in \reals$ with $\massflowNonBasis \equiv \massflowAtArc[n]$.
Employing a very similar construction as \cite[Chapter 6.1]{Gotzes2016}, we obtain the set $\eqSet$ of feasible $(\plossCoeff,\massflowAtArc[n])$-combinations and the associated cycle flow equation:
\begin{proposition}\label{prop:def-cycle-eq-h}
    Let \cref{asmp:cyclic-connected-graph} be satisfied.\\
    Then $\eqSet = \condSet{\plossCoeff \in \realsStrictPos^{\abs{A}}, \, \massflowAtArc[n] \in \reals}{h(\plossCoeff, \massflowAtArc[n]) = 0}$ with
    \begin{equation*}
        h(\plossCoeff, \massflowAtArc[n]) := - \sum_{\arc \in \arcs} f(\massflowAtArc[\arc]) = - \sum_{\arc \in \arcs} f(\massflowAtArc[n] - \beta_\arc)
    \end{equation*}
    and $\beta_\arc \in \reals$ for $\arc \in \arcs$.
    The constraint $h(\plossCoeff, \massflowAtArc[n]) = 0$ is the so-called \emph{cycle flow equation}.
\end{proposition}
Using $h$, a characterization of the set of all pressure loss coefficients $\plossCoeff$ which lead to the flow $\massflowAtArc[n]$ being bounded in some interval can be found:
\begin{lemma}\label{lemma:flow-in-intvl-as-subset-of-u}
    Let \cref{asmp:cyclic-connected-graph} be satified.
    Let $\lbMassflowAtArc[n]$, $\ubMassflowAtArc[n] \in \reals$, $\plossCoeff \in \realsStrictPos^{\abs{\arcs}}$ and let $h$ be as in \cref{prop:def-cycle-eq-h}.
    Then
    \begin{equation*}
         \condSet{\plossCoeff}{h(\plossCoeff, \massflowAtArc[n])=0 \text{ for some } \massflowAtArc[n] \in [\lbMassflowAtArc[n], \,\ubMassflowAtArc[n]]} = \condSet{\plossCoeff}{h(\plossCoeff, \ubMassflowAtArc[n]) \leq 0,\, h(\plossCoeff, \lbMassflowAtArc[n]) \geq 0}.
    \end{equation*}
\end{lemma}
\begin{proof}
    For constant $\phi \in \realsStrictPos^{\abs{\arcs}}$, the function $h(\plossCoeff, \massflowAtArc[n])$ is monotonically decreasing in $\massflowAtArc[n]$ since
    \begin{equation*}
        \frac{\d}{\d \massflowAtArc[n]}h(\plossCoeff, \massflowAtArc[n]) = -\sum_{i=0}^n \plossCoeffAtArc[i] \frac{\d}{\d \massflowAtArc[n]} f(\massflowAtArc[n] - \beta_i) = -\sum_{i=0}^n \plossCoeffAtArc[i] 2\abs{\massflowAtArc[n] - \beta_i} \leq 0.
    \end{equation*}
    Furthermore, $\lim_{\massflowAtArc[n] \rightarrow \pm \infty} h(\plossCoeff, \massflowAtArc[n]) = \mp \infty$.

    Let $A:=\condset{\plossCoeff \in \realsStrictPos^{\abs{\arcs}}}{h(\plossCoeff, \massflowAtArc[n])=0,\, \lbMassflowAtArc[n] \leq \massflowAtArc[n]}$ and $B := \condset{\plossCoeff \in \realsStrictPos^{\abs{\arcs}}}{h(\plossCoeff, \lbMassflowAtArc[n]) \geq 0}$.
    We show $A = B$ first:

    ``$\Rightarrow$'': Pick $\plossCoeff \in A$.
    By definiton of $A$, there is $\lbMassflowAtArc[n] \leq \massflowAtArc[n]$ with $h(\plossCoeff, \massflowAtArc[n]) = 0$.
    Since $h(\plossCoeff, \cdot)$ is monotonically decreasing, $h(\plossCoeff, \lbMassflowAtArc[n]) \geq h(\plossCoeff, \massflowAtArc[n]) = 0$.
    Therfore $\plossCoeff \in B$.

    ``$\Leftarrow$'': Pick $\plossCoeff \in B$.
    Since $h$ is continuous, $h(\plossCoeff, \lbMassflowAtArc[n]) \geq 0$, and $\lim_{\massflowAtArc[n] \to \infty} h(\plossCoeff, \massflowAtArc[n]) = -\infty$, the intermediate value theorem implies a $h(\plossCoeff, \massflowAtArc[n]) = 0$.
    Therefore $\plossCoeff \in A$.%

    This shows $A = B$.
    There is a similar result where the inequalities in the definitions of $A$, $B$ are flipped.
    Together, both results prove that
    \begin{equation*}
        \condSet{\plossCoeff \in \realsStrictPos^{\abs{\arcs}}}{h(\plossCoeff, \massflowAtArc[n])=0,\, \massflowAtArc[n] \in [\lbMassflowAtArc[n], \,\ubMassflowAtArc[n]]} = \condSet{\plossCoeff \in \realsStrictPos^{\abs{\arcs}}}{h(\plossCoeff, \ubMassflowAtArc[n]) \leq 0,\, h(\plossCoeff, \lbMassflowAtArc[n]) \geq 0}.
    \end{equation*}
\end{proof}
With this lemma, restricting $\massflowAtArc[n]$ to a given interval can be expressed equivalently by restricting the considered pressure loss coefficients $\plossCoeff$.
Furthermore, the constraints for $\plossCoeff$ are hyperplanes in $\reals^{\abs{\arcs}}$ as $h(\plossCoeff, \massflowAtArc[n])$ is linear in $\plossCoeff$.

We adapt a procedure from \cite[Proposition~5]{Gotzes2016} to our setting in order to identify intervals for the flow $\massflowAtArc[n]$ that guarantee constant flow direction on all arcs of the network.
Once the possible subsets are identified, we apply \cref{lemma:flow-in-intvl-as-subset-of-u} to relate the obtained flow intervals to subsets in the space of the uncertainty.\\
The absolute value functions only occur in the form $\plossCoeff \abs{\massflowAtArc[\arc]} (\massflowAtArc[\arc])$.
From \cref{prop:def-cycle-eq-h}, the flow $\massflowAtArc[\arc]$ along an arc $\arc \in \arcs$ is given by
\begin{equation*}
    \flowExtensionMap_\arc(\massflowAtArc[n]) = \massflowAtArc[n] - \beta_\arc.
\end{equation*}
Therefore, the absolute value $\abs{\flowExtensionMap_\arc(\massflowAtArc[n]) }$ can be eliminated by restricting the flow $\massflowAtArc[n]$ to either $\massflowAtArc[n] \geq \beta_\arc$ or $\massflowAtArc[n] \leq  \beta_\arc$.
Next, reorder $\beta_0, \beta_1,\ldots,\beta_n$ such that $\beta_{i_0} \leq \beta_{i_1} \leq \ldots \leq \beta_{i_n}$.
With this in mind, taking any consecutive pair $\beta_{i_{j}}, \beta_{i_{j+1}}$ yields an interval for $\massflowAtArc[n]$ such that the flow over the whole network is constant.
Due to \cite{Gotzes2016} and the nonzero demand from \cref{asmp:cyclic-connected-graph}, the solutions of $h(\plossCoeff, \massflowAtArc[n]) = 0$ can only be within $[\beta_{i_0}, \, \beta_{i_n}]$ for any fixed $\plossCoeff$.
Therefore, the absolute values can be eliminated by restricting $\massflowAtArc[n]$ to the intervals
\begin{align*}
    [\beta_{i_0}, \, \beta_{i_1}], && [\beta_{i_1}, \, \beta_{i_2}], && \ldots && [\beta_{i_{n-1}}, \, \beta_{i_n}].
\end{align*}
Applying \cref{lemma:flow-in-intvl-as-subset-of-u} to these intervals yields an equivalent condition for constant flow directions in the space of the uncertainty.
\begin{proposition}
    Let \cref{asmp:cyclic-connected-graph} be satisfied and let
    \begin{equation*}
        \uncertaintySet_j := \uncertaintySet \cap \condSet{\plossCoeff \in \realsStrictPos^{\abs{\arcs}}}{h(\plossCoeff, \beta_{i_{j+1}}) \leq 0,\, h(\plossCoeff, \beta_{i_{j}}) \geq 0} \quad \text{for } j=0,\ldots n-1.
    \end{equation*}
    Then the set containment question $\uncertaintySet \subseteq \proj_\plossCoeff(\feasCoeffsAndFlows) = \proj_\plossCoeff(\eqSet_\uncertaintySet \cap \inEqSet)$ can be decided by solving the subproblems
    \begin{equation*}
        \uncertaintySet_j \subseteq \proj_\plossCoeff(\eqSet_{\uncertaintySet_j} \cap \inEqSet) \quad \text{for } j=0,\ldots n-1.
    \end{equation*}
\end{proposition}
We remark that if $\uncertaintySet$ is polyhedral then $\uncertaintySet_j$ is polyhedral as well.

\section{Numerical Experiments}\label{sec:numerical-experiments}
In this section, some practical results of the feasibility and infeasibility approaches on a set of small gas networks under uncertainty are presented.
Instead of considering arbitrary gas networks, we focus on highlighting our methods' performance on the core problem: deciding a single cycle under uncertainty.
Using \cref{lemma:robust-tree-psqr-interval}, the feasibility of any subtree in a given network can be reduced if the pressure at the root node is contained in a pre-calculated interval.
This allows us to remove any subtree by updating the pressure bounds at the intersecting node with the remaining network.
Assuming there is only one remaining cycle, \cref{lemma:flow-in-intvl-as-subset-of-u} is then used to split the problem into subproblems on subsets of the uncertainty set while eliminating all absolute values.
Since this just increases the number of problems to consider but does not fundamentally change their nature, we start with a single cycle and uncertainty sets that guarantee constant flow direction on all arcs.

The example networks are cyclic with nodes $V = \{1, 2, \ldots, n\}$ for $n \in \{2, \ldots, 7\}$ and arcs $A = \{(1, 2), (2, 3), \ldots, (n-1, n), (n, 1)\}$.
A family of uncertainty sets is considered:
\begin{equation*}
    \uncertaintySet(c) = \times_{a \in A} [1, c],\quad c\in[2, 4].
\end{equation*}
Furthermore, define two special uncertainty sets,
\begin{equation*}
    \uncertaintySet_\text{feas} := \uncertaintySet(2) \quad\text{ and }\quad \uncertaintySet_\text{infeas} := \uncertaintySet(4),
\end{equation*}
which we want to investigate with respect to feasibility and infeasibility, respectively.

\Cref{tab:instance-data} shows the parameters of the considered instances.
The columns denote the nodes within the network.
Each row denotes the specific instance with $n$ nodes.
Within each row, the demand and bounds of the squared pressure $\pressSqr$ at each node is displayed in the first and second lines, respectively.
\begin{table}[H]
  \centering
  \caption{Demand and squared pressure bounds $\pressSqr$ per node $\node$ for each test network.}
  \label{tab:instance-data}
  \begin{adjustbox}{max width=\linewidth}
\begin{tabular}{ll|rrrrrrr}
\toprule
    &              &           &             & node $\node \in V$ &      &           &           &           \\
    &              &         1 &           2 &           3 &           4 &         5 &         6 &         7 \\
\midrule
n=2 & demand       &       -10 &          10 &             &             &           &           &           \\
    & $\pressSqr$-bounds &  [0, 200] &  [140, 200] &             &             &           &           &           \\
  \midrule
n=3 & demand       &       -10 &           2 &           8 &             &           &           &           \\
    & $\pressSqr$-bounds &  [0, 200] &    [0, 200] &  [130, 200] &             &           &           &           \\
  \midrule
n=4 & demand       &       -10 &           2 &           6 &           2 &           &           &           \\
    & $\pressSqr$-bounds &  [0, 200] &    [0, 200] &  [115, 200] &    [0, 200] &           &           &           \\
  \midrule
n=5 & demand       &       -10 &           1 &           1 &           6 &         2 &           &           \\
    & $\pressSqr$-bounds &  [0, 200] &    [0, 200] &    [0, 200] &  [100, 200] &  [0, 200] &           &           \\
  \midrule
n=6 & demand       &       -10 &           1 &           1 &           6 &         1 &         1 &           \\
    & $\pressSqr$-bounds &  [0, 200] &    [0, 200] &    [0, 200] &   [70, 200] &  [0, 200] &  [0, 200] &           \\
  \midrule
n=7 & demand       &       -10 &           1 &           1 &           1 &         4 &         2 &         1 \\
    & $\pressSqr$-bounds &  [0, 200] &    [0, 200] &    [0, 200] &    [0, 200] & [50, 200] &  [0, 200] &  [0, 200] \\
\bottomrule
\end{tabular}

  \end{adjustbox}
\end{table}
Every network's $\inEqSet$-set (see Subsection~\ref{sec:gas-definitions}) is made up of $n(n-1)$ inequalities $\genIneqFunc_i$  $(i \in I)$.
Each inequality is checked for \emph{feasibility} using \eqref{prob:minconstraints}; all inequalities are checked at once for \emph{infeasibility} using \eqref{prob:poly-sep-proj}.
Both optimization tasks are solved using SDP relaxations of the problems.
We remark that \eqref{prob:poly-sep-proj} could be applied to all constraints individually.
However, experiments show that solving the problem for a single constraint individually is only marginally faster than solving the problem for all constraints at once.
Therefore, we solve the infeasibility problem once with all constraints combined rather than up to $\abs{I}$ subproblems by considering each constraint on its own.

% The example networks are constructed so that the direction of the gas flow is constant for all elements of the uncertainty sets.
% All gas enters the network at node 1 and is transported towards the highest demand node.
% This allows us to fix the flow directions and remove the absolute values in all subsequent computations.
All experiments were carried out on a notebook with four Intel i7-4810MQ cores running at 2.80GHz each and 16 GB of RAM.
The methods were implemented using MATLAB R2016b.
GloptiPoly 3.8 \cite{Henrion2009} was used for the feasibility models since it provides a straight forward interface for solving polynomial optimization problems.
Since the infeasibility method exceeds the capabilities of GloptiPoly, this approach was implemented using the SOS-module of YALMIP R20160930 \cite{Loefberg2009}.
The resulting SDP problems were solved with MOSEK 8 \cite{Mosek2011} using 4 threads.

Some problems were not solvable with the desired precision.
This happened although we evaluated the problems on a variety of solvers including SeDuMi \cite{Sturm1999} and SDPT3 \cite{Toh1999} as well as on a third \modeling tool, SOSTOOLS \cite{Papachristodoulou2013}.
The chosen combination of MOSEK with GloptiPoly and YALMIP offered the most robust behavior amongst all considered possibilities.

\subsection{Effectiveness of the Methods}
The effectiveness of both methods can be measured in the typical \runningtimes of the semidefinite subproblems as well as in hierarchy level at which set containment can be decided.

First, the results of both methods on a fixed network are presented.
\Cref{tab:threenodes-both-methods} shows the outcome of both methods for the $n=3$ instance over $\uncertaintySet_\text{infeas}$.
The columns are separated into groups concerning the feasibility method \cref{prob:minconstraints} and the infeasibility method \cref{prob:poly-sep-proj} with a further distinction into the employed hierarchy level.
The rows in the feasibility part denote the constraint $\genIneqFunc_i$ which is minimized.
Since the infeasibility method is applied to all constraints at once, there is only one row of results in the infeasibility part of the table.
Cells marked by ``-'' indicate numerical difficulties, \ie we were unable to solve the specific problem to the desired precision. \\
The feasibility approach has a positive objective for five out of six subproblems, thus confirming set containment for those constraints.
Out of these five problems, four were decided on the second hierarchy level while one required a level 3 solution.
When applying the infeasibility approach, the level 3 model is unbounded, thereby refuting set containment.
Over all, the instance therefore isn't robust feasible.
\begin{table}
    \centering
    \caption{Objectives of the feasibility method solving \cref{prob:minconstraints} and infeasibility method solving \cref{prob:poly-sep-proj} for the three node instance over $\uncertaintySet_\text{infeas}$. Each row in the feasibility group denotes the subproblem with objective function $\genIneqFunc_i$.}
    \label{tab:threenodes-both-methods}
    \begin{adjustbox}{max width=\linewidth}
\begin{tabular}{r|R{1.5cm}R{1.5cm}||R{1.5cm}R{1.5cm}}
\toprule
  & \multicolumn{2}{c||}{feasibility}            & \multicolumn{2}{c}{infeasibility}          \\
i &        level 2                               &       level 3                              & level 2 &  level 3  \\
\midrule
1 & \textcolor{black}{\textbf{ 216.89}} &  217.39                                    & \multirow{6}{*}{0.00} & \multirow{6}{*}{\textcolor{black}{\textbf{unbnd}}} \\
2 & \textcolor{black}{\textbf{  53.09}} &   -                                        &                      &   \\
3 & \textcolor{black}{\textbf{ 228.63}} &  228.63                                    &                      &   \\
4 &                   -116.79                    &                      -                     &                      &   \\
5 & \textcolor{black}{\textbf{ 201.67}} &    -                                       &                      &                     \\
6 &                    -35.99                    & \textcolor{black}{\textbf{20.34}} &                      &           \\
\bottomrule
\end{tabular}

    \end{adjustbox}
\end{table}

Next, the required levels of the relaxation hierarchy are evaluated.
For this purpose, each constraint of each instance is considered for set containment while gradually increasing the hierarchy level from two to four.
Once a subproblem is solved successfully, the corresponding number of solved problems on this specific level is incremented in the table.

\Cref{tab:success-per-level-ufeas} contains the feasibility methods' results for all instances on the smaller uncertainty set $\uncertaintySet_\text{feas}$.
Each row denotes the considered instance with $n$ nodes and a total of $\abs{I}$ subproblems.
The columns indicate how many of the feasibility problems were solved successfully on the respective level.
For any subproblem, only the first success is counted, thus the sum of each row can be at most $\abs{I}$.
If the row-wise sum is less then $\abs{I}$, this implies that some problems were not solvable with the desired precision.

It can be observed that the feasibility approach almost exclusively confirms set containment at the second level.
At most one subproblem per instance required solving of a level 3 problem.
As suspected, all instances are robust feasibly with this uncertainty region.

\begin{table}
    \centering
    \caption{For a given instance with $n$ nodes, count how many subproblems out of $I$ were solved successfully using the feasibility method.
             Positive outcomes of each subproblem are counted only once on the smallest level.
             All instances were solved over the $\uncertaintySet_\text{feas}$ uncertainty set.}
    \label{tab:success-per-level-ufeas}
    \begin{adjustbox}{max width=\linewidth}
\begin{tabular}{cc||R{1.5cm}R{1.5cm}R{1.5cm}}
\toprule
n        & $\abs{I}$ &  \multicolumn{1}{c}{level 2} &    \multicolumn{1}{c}{level 3}      & \multicolumn{1}{c}{level 4}  \\
\midrule
2        &         2 &   1 &    1     & 0    \\
3        &         6 &   5 &    1     & 0    \\
4        &        12 &  11 &    1     & 0    \\
5        &        20 &  19 &    1     & 0    \\
6        &        30 &  29 &    1     & 0    \\
7        &        42 &  42 &    0     & 0    \\
\bottomrule
\end{tabular}

    \end{adjustbox}
\end{table}

Using the larger uncertainty set $\uncertaintySet_\text{infeas}$, both the feasibility and the infeasibility method were applied to all instances.
\Cref{tab:success-per-level-uinfeas} summarizes all results.
Each row denotes the considered instance with $n$ nodes.
The columns are separated into groups according to the employed method with further distinction for the used hierarchy level.
Each column in the feasibility group indicates how many of the feasibility problems were solved successfully.
For any subproblem, only the first success is counted, therefore the sum of each row in the feasibility group can be at most $\abs{I}$.
The columns in the infeasibility group denote the status of the corresponding problem.
Cells marked with``zero obj.'' indicate global optimality of the considered problem but an objective value of zero, which is insufficient to show certify infeasibility.
Cells marked with a checkmark ($\checkmark$) represent an unbounded objective and thus a negative answer to the set containment question.
As usual, ``-'' marks numerical difficulties.\\
Many feasibility problems were solved successfully at the second hierarchy level.
Set containment of some constraints could not be confirmed with the feasibility method using the given levels.
This is either due to numerical problems or negative objective values.
However, for almost all instances, the infeasibility method was able to provide a certificate against set containment using the third hierarchy level relaxation.
This shows that $\uncertaintySet_\text{infeas}$ is robust infeasible for the $n=2,\ldots,6$ instances.
\begin{table}
    \centering
    \caption{For a given instance with $n$ nodes, count how many subproblems of $I$ were solved successfully using the feasibility method.
             For each subproblem, a positive outcome is counted only once on the smallest level.
             The results of the infeasibility method are displayed in the right column group.
             All instances were solved over the $\uncertaintySet_\text{infeas}$ uncertainty set.}
    \label{tab:success-per-level-uinfeas}
    \begin{adjustbox}{max width=\linewidth}
% \begin{tabular}{cc||R{1.5cm}R{1.5cm}R{1.5cm}R{1.5cm}|R{1.5cm}R{1.5cm}R{1.5cm}R{1.5cm}}
% \toprule
%          &           & \multicolumn{4}{c|}{minimize} &\multicolumn{4}{c}{integral}   \\
% n        & $\abs{I}$ &  2 &    3     & 4 & miss     & 2 &       3 & 4 &   miss \\
% \midrule
% 2        &         2 &   1 &    0     & 0 &    1     & 0 &       1 & 0 &       1 \\
% 3        &         6 &   4 &    1     & 0 &    1     & 0 &       1 & 0 &       5 \\
% 4        &        12 &   9 &    1     & 0 &    2     & 0 &       1 & 0 &      11 \\
% 5        &        20 &  16 &    1     & 0 &    3     & 0 &       1 & 0 &      19 \\
% 6        &        30 &  25 &    1     & 0 &    4     & 0 &       1 & 0 &      29 \\
% 7        &        42 &  36 &    2     & 0 &    4     & 0 &       0 & 0 &      42 \\
% \bottomrule
% \end{tabular}
\begin{tabular}{c||c|R{2.5cm}R{2.5cm}R{2.5cm}||R{2.5cm}R{2.5cm}R{2.5cm}}
\toprule
         &            \multicolumn{4}{c||}{feasibility} &\multicolumn{3}{c}{infeasibility}   \\
n        & $\abs{I}$ &  \multicolumn{1}{c}{level 2} &    \multicolumn{1}{c}{level 3}      & \multicolumn{1}{c||}{level 4}  & \multicolumn{1}{c}{level 2} &       \multicolumn{1}{c}{level 3} & \multicolumn{1}{c}{level 4} \\
\midrule
2        &         2 &   1 &    0     & 0  & zero obj. &       \textcolor{black}{$\checkmark$} & \textcolor{black}{$\checkmark$}  \\
3        &         6 &   4 &    1     & 0  & zero obj. &       \textcolor{black}{$\checkmark$} & \textcolor{black}{$\checkmark$}  \\
4        &        12 &   9 &    1     & 0  & zero obj. &       \textcolor{black}{$\checkmark$} & \textcolor{black}{$\checkmark$}  \\
5        &        20 &  16 &    1     & 0  & zero obj. &       \textcolor{black}{$\checkmark$} & \textcolor{black}{$\checkmark$}  \\
6        &        30 &  25 &    1     & 0  & zero obj. &       \textcolor{black}{$\checkmark$} & \textcolor{black}{$\checkmark$}  \\
7        &        42 &  36 &    2     & 0  & zero obj. &       - & -  \\
\bottomrule
\end{tabular}

    \end{adjustbox}
\end{table}

To conclude this set of test runs, \Cref{tab:feasmethod-runtimes,tab:infeasmethod-runtimes} show the characteristic \runtimes where each row denotes the $n$-node instance.
For the feasibility approach, the columns show mean \runtime and standard deviation using the specific relaxation hierarchy level.
All values are aggregated over all subproblems of the given instance and hierarchy level.
Since the infeasibility approach is a single problem when instance and hierarchy level are fixed, no aggregation is possible and we show the \runtime as-is.
It can be observed that the \runtimes are quite small for the level 2 problems but increase quickly for higher levels and larger instances.
\begin{table}
    \centering
    \caption{Mean and standard deviation of the feasibility method's \runtime on $\uncertaintySet_\text{feas}$.
             Each row shows the aggregated values for all subproblems of the $n$-node instance per hierarchy level.}
    \label{tab:feasmethod-runtimes}
    \begin{adjustbox}{max width=\linewidth}
\begin{tabular}{c||R{1.5cm}R{1.5cm}|R{1.5cm}R{1.5cm}|R{1.5cm}R{1.5cm}}
\toprule
   & \multicolumn{2}{c|}{level 2} & \multicolumn{2}{c|}{level 3}        & \multicolumn{2}{c}{level 4} \\
n &      \multicolumn{1}{c}{mean} &       \multicolumn{1}{c|}{std} &       \multicolumn{1}{c}{mean} &       \multicolumn{1}{c|}{std} &        \multicolumn{1}{c}{mean} &        \multicolumn{1}{c}{std} \\
\midrule
2     &  \SI{0.032}{\second} &  \SI{0.019}{\second} &   \SI{0.042}{\second} &  \SI{0.014}{\second} &    \SI{0.111}{\second} &   \SI{0.021}{\second} \\
3     &  \SI{0.040}{\second} &  \SI{0.010}{\second} &   \SI{0.111}{\second} &  \SI{0.070}{\second} &    \SI{0.605}{\second} &   \SI{0.083}{\second} \\
4     &  \SI{0.048}{\second} &  \SI{0.015}{\second} &   \SI{0.324}{\second} &  \SI{0.030}{\second} &    \SI{3.899}{\second} &   \SI{0.187}{\second} \\
5     &  \SI{0.083}{\second} &  \SI{0.024}{\second} &   \SI{1.229}{\second} &  \SI{0.185}{\second} &   \SI{26.679}{\second} &   \SI{3.040}{\second} \\
6     &  \SI{0.147}{\second} &  \SI{0.047}{\second} &   \SI{4.533}{\second} &  \SI{0.894}{\second} &  \SI{148.397}{\second} &   \SI{9.500}{\second} \\
7     &  \SI{0.241}{\second} &  \SI{0.061}{\second} &  \SI{15.721}{\second} &  \SI{2.328}{\second} &  \SI{809.944}{\second} &  \SI{71.564}{\second} \\
\bottomrule
\end{tabular}

    \end{adjustbox}
\end{table}
\begin{table}
    \centering
    \caption{Runtime of the infeasibility method on $\uncertaintySet_\text{infeas}$ where
             each row denotes the $n$-node instance and each column the respective level.}
    \label{tab:infeasmethod-runtimes}
    \begin{adjustbox}{max width=\linewidth}
\begin{tabular}{c||R{1.5cm}R{1.5cm}R{1.5cm}}\toprule
n &    level 2              &        level 3           &        level 4            \\
\midrule
2 &  \SI[round-mode=places,round-precision=3]{0.415459}{\second} &  \SI[round-mode=places,round-precision=3]{ 0.564157}{\second} &  \SI[round-mode=places,round-precision=3]{  0.504456}{\second} \\
3 &  \SI[round-mode=places,round-precision=3]{0.435395}{\second} &  \SI[round-mode=places,round-precision=3]{ 0.501112}{\second} &  \SI[round-mode=places,round-precision=3]{  0.849461}{\second} \\
4 &  \SI[round-mode=places,round-precision=3]{0.433069}{\second} &  \SI[round-mode=places,round-precision=3]{ 0.770711}{\second} &  \SI[round-mode=places,round-precision=3]{  3.908801}{\second} \\
5 &  \SI[round-mode=places,round-precision=3]{0.389520}{\second} &  \SI[round-mode=places,round-precision=3]{ 1.854349}{\second} &  \SI[round-mode=places,round-precision=3]{ 20.202790}{\second} \\
6 &  \SI[round-mode=places,round-precision=3]{0.431818}{\second} &  \SI[round-mode=places,round-precision=3]{ 4.914542}{\second} &  \SI[round-mode=places,round-precision=3]{134.531288}{\second} \\
7 &  \SI[round-mode=places,round-precision=3]{0.643027}{\second} &  \SI[round-mode=places,round-precision=3]{10.280900}{\second} &  \SI[round-mode=places,round-precision=3]{975.160599}{\second} \\
\bottomrule
\end{tabular}

%  \SI{0.010}{\second}

    \end{adjustbox}
\end{table}

\subsection{Evaluation of the Gap Between Methods}
The proposed methods are based on semidefinite relaxations of polynomial problems (see \cref{subsec:poly-methods}).
Since the objective values of relaxed problems are smaller or equal than the non-relaxed optimal values (for minimization problems), it is expected that the feasibility and infeasibility approach can decide a smaller number of problems than their non-relaxed counterparts.
The aim of this section is to investigate how large the ``gap'' between feasibility and infeasibility approach is.
After fixing a hierarchy level, all problems which cannot be decided by either feasibility or infeasibility approach are said to fall into this relaxation gap.
In order to compare both methods, we need to apply the infeasibility approach to the same constraint as the feasibility method.
This is different to all previous tests where the infeasibility method was solved for all constraints at once.

Consider the parameterized uncertainty set $\uncertaintySet(c)$ for increasing $c \in [2, 4]$.
From \cref{tab:success-per-level-ufeas}, it can be derived that all subproblems are feasible for $\uncertaintySet_\text{feas} = \uncertaintySet(2)$.
On the other hand, as \cref{tab:success-per-level-uinfeas} shows, all instances are infeasible for the larger $\uncertaintySet_\text{infeas} = \uncertaintySet(4)$.
This implies that there is always at least one violated constraint $\genIneqFunc_i$ when using $\uncertaintySet(4)$.

For this \testset, we select one subproblem per instance that is infeasible for the larger uncertainty set.
Then, the feasibility and infeasibility approaches are solved for the selected subproblems over all twenty uncertainty sets $\uncertaintySet(c)$ for $c = 2 + i\frac{1}{10}$, $i=0,\ldots,20$.
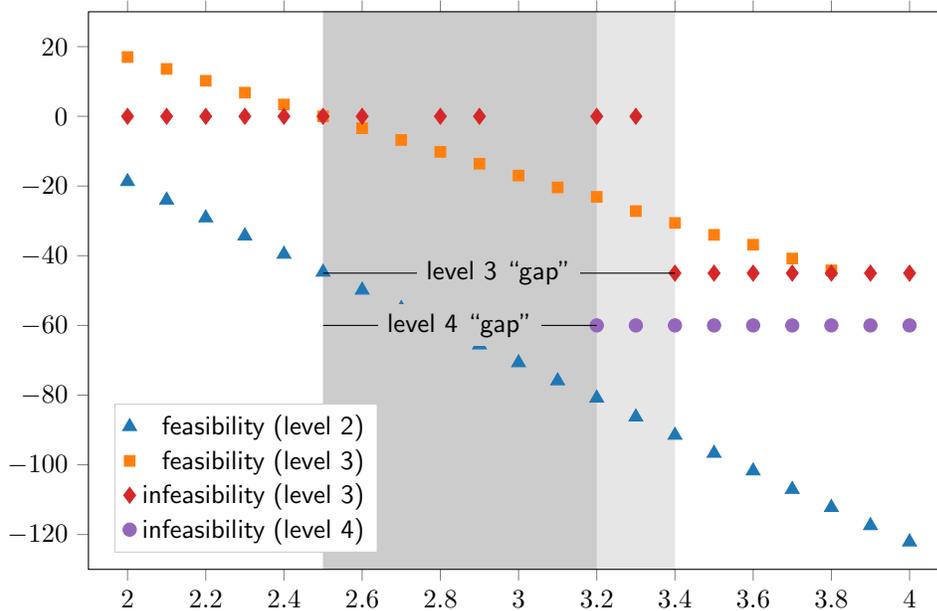
\begin{figure}
    \centering
    \caption{Objective values of the two methods for varied $c\in\{2.0,2.1,\ldots, 4\}$ on the four node instance.}
    \label{fig:method-gap-4nodes}
%!TEX root = ../../poly-set-containment-gas.tex

\begin{tikzpicture}[font=\sffamily]
% \pgfplotsset{
%   set layers,% using layers
%   mark layer=axis tick labels% defines the layer of the marks
% }
\pgfplotsset{set layers}

\definecolor{color0}{rgb}{0.12156862745098,0.466666666666667,0.705882352941177}
\definecolor{color3}{rgb}{0.83921568627451,0.152941176470588,0.156862745098039}
\definecolor{color1}{rgb}{1,0.498039215686275,0.0549019607843137}
\definecolor{color4}{rgb}{0.580392156862745,0.403921568627451,0.741176470588235}
\definecolor{color2}{rgb}{0.172549019607843,0.627450980392157,0.172549019607843}

\begin{axis}[set layers,
xmin=1.9, xmax=4.1,
ymin=-130, ymax=30,
height=9cm,
width=\textwidth,
tick align=outside,
x grid style={white!69.019607843137251!black},
y grid style={white!69.019607843137251!black},
legend entries={{feasibility (level 2)},{feasibility (level 3)},{infeasibility (level 3)},{infeasibility (level 4)}},
legend cell align={right},
legend style={draw=white!80.0!black},
legend pos=south west,
axis on top
]
\addplot [only marks, mark size=3pt,mark=triangle*, draw=color0, fill=color0, colormap/viridis]
table {%
x                      y
+2.000000000000000e+00 -1.872616764128380e+01
+2.100000000000000e+00 -2.402950470273320e+01
+2.200000000000000e+00 -2.913778183792030e+01
+2.300000000000000e+00 -3.432225474442940e+01
+2.400000000000000e+00 -3.957718556619621e+01
+2.500000000000000e+00 -4.471789876135110e+01
+2.600000000000000e+00 -4.987507497484321e+01
+2.700000000000000e+00 -5.518224924353420e+01
+2.800000000000000e+00 -6.028909259538511e+01
+2.900000000000000e+00 -6.559882227738980e+01
+3.000000000000000e+00 -7.070876545449708e+01
+3.100000000000000e+00 -7.590004062972730e+01
+3.200000000000000e+00 -8.085930226132150e+01
+3.300000000000000e+00 -8.627773659165810e+01
+3.400000000000000e+00 -9.152256390877990e+01
+3.500000000000000e+00 -9.669671165420701e+01
+3.600000000000000e+00 -1.017509181707200e+02
+3.700000000000000e+00 -1.070499030646880e+02
+3.800000000000000e+00 -1.122368506512950e+02
+3.900000000000000e+00 -1.174545533996090e+02
+4.000000000000000e+00 -1.221336222692290e+02
};
\addplot [only marks, mark size=2pt, mark=square*, draw=color1, fill=color1, colormap/viridis]
table {%
x                      y
+2.000000000000000e+00 +1.700035326752310e+01
+2.100000000000000e+00 +1.360004208305070e+01
+2.200000000000000e+00 +1.020000208579250e+01
+2.300000000000000e+00 +6.800071780615990e+00
+2.400000000000000e+00 +3.400001653184920e+00
+2.500000000000000e+00 +1.207869742287930e-04
+2.600000000000000e+00 -3.399989070496280e+00
+2.700000000000000e+00 -6.799396533713850e+00
+2.800000000000000e+00 -1.019999993759820e+01
+2.900000000000000e+00 -1.359999994768100e+01
+3.000000000000000e+00 -1.699999926226810e+01
+3.100000000000000e+00 -2.039991845214580e+01
+3.200000000000000e+00 -2.310615707555750e+01
+3.300000000000000e+00 -2.719998668688780e+01
+3.400000000000000e+00 -3.059999408222630e+01
+3.500000000000000e+00 -3.399999470268520e+01
+3.600000000000000e+00 -3.683742028024960e+01
+3.700000000000000e+00 -4.079980995212840e+01
+3.800000000000000e+00 -4.419961117310200e+01
};

\addplot [only marks, mark size=3pt, mark=diamond*, draw=color3, fill=color3, colormap/viridis]
table {%
x                      y
+2.000000000000000e+00 -7.886832246820370e-11
+2.100000000000000e+00 -4.893217453515000e-09
+2.200000000000000e+00 -2.698246889917990e-09
+2.300000000000000e+00 -4.750759069233230e-09
+2.400000000000000e+00 -6.559857025864610e-10
+2.500000000000000e+00 -1.750281887068390e-08
+2.600000000000000e+00 -2.798942117031130e-08
+2.800000000000000e+00 -1.020718320933130e-07
+2.900000000000000e+00 -5.802299402166320e-08
+3.200000000000000e+00 -4.107777210691360e-07
+3.300000000000000e+00 -5.546945469586670e-06
+3.400000000000000e+00 -4.500000000000000e+01
+3.500000000000000e+00 -4.500000000000000e+01
+3.600000000000000e+00 -4.500000000000000e+01
+3.700000000000000e+00 -4.500000000000000e+01
+3.800000000000000e+00 -4.500000000000000e+01
+3.900000000000000e+00 -4.500000000000000e+01
+4.000000000000000e+00 -4.500000000000000e+01
};
\addplot [only marks, mark size=2.5pt,mark=*, draw=color4, fill=color4, colormap/viridis]
table {%
x                      y
+3.200000000000000e+00 -6.000000000000000e+01
+3.300000000000000e+00 -6.000000000000000e+01
+3.400000000000000e+00 -6.000000000000000e+01
+3.500000000000000e+00 -6.000000000000000e+01
+3.600000000000000e+00 -6.000000000000000e+01
+3.700000000000000e+00 -6.000000000000000e+01
+3.800000000000000e+00 -6.000000000000000e+01
+3.900000000000000e+00 -6.000000000000000e+01
+4.000000000000000e+00 -6.000000000000000e+01
};
\begin{pgfonlayer}{main}
\addplot[fill=black!10!white,draw=none] coordinates {(2.5,50) (3.4,50) (3.4,-200) (2.5,-200) } \closedcycle;
\addplot[fill=black!20!white,draw=none] coordinates {(2.5,50) (3.2,50) (3.2,-200) (2.5,-200) } \closedcycle;
\end{pgfonlayer}

% \draw[->](axis cs:2.5,-45)--(axis cs:3.2,-45);
\begin{pgfonlayer}{axis foreground}
% \addplot[<->, thick] coordinates
%             {(2.5,-60) (2.85,-60) (3.2,-60)} node[above, pos=0.5] {level 4 ``gap''};
% \addplot[<->, thick] coordinates
%             {(2.5,-45) (2.95, -45) (3.4,-45)} node[above, pos=0.5] {level 3 ``gap''};

\draw (axis cs:2.5,-45) -- (axis cs:3.4,-45) node [midway,fill=black!20!white] {level 3 ``gap''};
\draw (axis cs:2.5,-60) -- (axis cs:3.2,-60) node [midway,fill=black!20!white] {level 4 ``gap''};
\end{pgfonlayer}
\end{axis}

\end{tikzpicture}

\end{figure}
\Cref{fig:method-gap-4nodes} shows the results in more detail for the four node instance.
We consider the subproblem that is marked as infeasible in \cref{tab:success-per-level-uinfeas}.
The objective values of the feasibility problem \cref{prob:minconstraints} are marked with blue (level 2) and orange (level 3) triangles in the figure.
Additionally, the values of solving \cref{prob:poly-sep-proj} are marked using red (level 3) and purple (level 4) circles.
We remark that the outcome of the infeasibility method for level 2 is omitted since as all subproblems were feasible but had objective value of zero.
Unbounded subproblems of the infeasibility method are marked with an objective value of fifteen times their level.
Missing data points can be attributed to numerical difficulties of the SDP solver.

As can be observed, no instance can be decided on the second hierarchy level since all solutions of the feasibility method have negative objective values and all solutions of the infeasibility method have objective value zero (not shown in the figure).
On the third hierarchy level, the feasibility approach confirms set containment for $c \in \{2.0, 2.1, \ldots, 2.5\}$ as these problems have positive objective value.
With the same level, the infeasibility approach finds certificates against set containment for $c \in \{3.4,\ldots, 4.0\}$.
For the problems with $c\in \{2.6, \ldots, 3.3\}$, neither of the methods was able to decide set containment successfully (disregarding numerical difficulties).
In this range, the feasibility method only returns negative objective values and all objective  values of the infeasibility method were zero.

% Therefore, the gap on this level is said to be $0.9$.
Increasing the hierarchy level to four leads to numerical problems for all feasibility models, but also increases the number of successfully solved infeasibility models by two ($c=3.2$ and $c=3.3$).
This confirms the expectation that increasing the hierarchy level can lead to more certificates for non-set containment.

The results over all instances is summarized in \cref{tab:ugrowth-gap}.
For each hierarchy level, it shows both he largest value for $c$ (indicated by $c_\text{feas}$) such that the feasibility approach confirms set containment and the smallest value for $c$ (indicated by $c_\text{infeas}$) where a certificate for infeasibility could be obtained.
Note that these bounds on $c$ take all smaller hierarchy levels into account as well.
The gap column is the difference $c_\text{infeas}-c_\text{feas}$ and indicates the range of problems which could not be solved successfully with either feasibility and infeasibility approach.
Again it can be observed that the gap is reduced after increasing the hierarchy level as this leads to a tighter relaxation for the feasibility approach and admits a richer set of polynomials for the infeasibility certificate.
\begin{table}
    \centering
    \caption{Extreme values for $c$ where the feasibility ($c_\text{feas}$) and infeasibility ($c_\text{infeas}$) methods can solve the problem.}
    \label{tab:ugrowth-gap}
    \begin{adjustbox}{max width=\linewidth}
\begin{tabular}{c||R{1.5cm}R{1.5cm}R{1.5cm}|R{1.5cm}R{1.5cm}R{1.5cm}}
\toprule
      & \multicolumn{3}{c|}{level 3}        & \multicolumn{3}{c}{level 4} \\
n     & $c_\text{feas}$ & $c_\text{infeas}$ & gap  & $c_\text{feas}$ & $c_\text{infeas}$ & gap \\
\midrule
2     &             2.4 &               3.3 &  0.9 &             2.4 &               2.9 &  0.5 \\
3     &             2.4 &               3.2 &  0.8 &             2.4 &               3.1 &  0.7 \\
4     &             2.5 &               3.4 &  0.9 &             2.5 &               3.2 &  0.7 \\
5     &             2.4 &               3.4 &  1.0 &             2.4 &               3.3 &  0.9 \\
6     &             2.6 &               3.7 &  1.1 &             3.1 &               3.6 &  0.5 \\
7     &             3.0 &                   &      &             3.3 &                   &      \\
\bottomrule
\end{tabular}

    \end{adjustbox}
\end{table}

\section{Concluding Remarks}\label{sec:conclusion}
In this paper, we study feasibility and infeasibility of nonlinear two-stage fully adjustable robust feasibility problems with an empty first stage.
We propose to solve this problem by deciding whether the given uncertainty set is a subset of the projection of all feasible (uncertainty, solution)-pairs.
A particular challenge with this approach is given by the projected set whose defining constraints are typically not available.
Compared to typical methods from robust optimization, our approach requires no additional restrictions such as like convexity of the problem or the uncertainty set.
Furthermore, it can decide the fully adjustable problem without using (possible approximative) decision rules for the second stage variables.
We develop two approaches towards solving this problem, one for deciding feasibility and one for deciding infeasibility.
As we solve relaxations of the proposed methods in practice, two distinct methods are necessary since a single method cannot be expected to solve both sides of the question.
The first approach for deciding infeasibility uses a separation argument to find polynomial that certifies violation of the set containment question.
The second approach is based the assumption that part of the problem constraints define a unique solution for a fixed element of the uncertainty set.
Exploiting this fact allows a reformulation as a set containment question over two regular (non-projected) sets.
Set containment can then be confirmed by minimizing the constraint functions of the superset over the subset.
In our setting, both methods lead to polynomial optimization problems.
For solving the polynomial problems in practice, we fall back onto the Lasserre SDP relaxation hierarchy.

The proposed models are then applied to an uncertain gas transport problem.
This is a non-convex quadratic problem with absolute value functions.
First, we show how this problem can be decided exactly on tree structured using LP duality to decide set containment of polyhedra.
Next, this result is used to preprocess larger problems so that only cycles remain.
Lastly, we present different ideas how to remove the absolute values functions from the problem formulation.
By removing the absolute values, the problem is transformed to a purely polynomial description to which the proposed methods can be applied.

Both approaches are then solved on a set of cyclic test networks.
For problems where deciding robustness was possible, we observe that typically level 2 or level 3 of the Lasserre hierarchy were sufficient.
We further investigate the strength of the relaxation by searching for uncertainty sets where neither feasibility nor infeasibility can be decided for a given instance and hierarchy level.
As can be expected, increasing the level yields tighter relaxations which translates into a more effective method.

As an outlook, the developed ideas could be applied to similar potential driven network flow problems such as \eg the DC optimal power problem flow or water network problems.
Concerning the application to gas networks, extending the relation between subsets of the uncertainty set and flow directions to networks with multiple intermeshed cycles is another relevant question.
Lastly, using the feasibility methods as part of a larger two-stage robust optimization task with non empty first stage provides another possible extension of the studied problem.
In case of gas, first stage variables model decisions of the network operator \eg the compressor machines' power level.

\appendix

\section*{Acknowledgments}
The authors thank Prof.~Dick~den~Hertog for fruitful discussions on the topic.
Furthermore, the authors would like to thank the anonymous reviewers
for their valuable comments and insightful suggestions that improved the quality
of the paper.
\bibliographystyle{siamplain}
\bibliography{references}
\end{document}